\documentclass[11pt]{article}
\usepackage{bm,amsfonts,amssymb,amsthm,amsmath}
\usepackage{mathrsfs}
\usepackage{graphicx,color}
\usepackage{CJK}
\usepackage{indentfirst,enumerate}

\newcommand{\md}{\mathrm{d}}

\newtheorem{theorem}{Theorem}
\numberwithin{theorem}{section}

\newtheorem{rem}[theorem]{Remark}
\numberwithin{equation}{section}

\definecolor{orange}{rgb}{1,0.5,0}
\definecolor{rb}{rgb}{1,0,1}
\allowdisplaybreaks[4]

\begin{document}

\title{\textbf{Unconditionally positivity preserving and energy dissipative schemes for Poisson--Nernst--Planck equations}\footnote{This work is supported in part by AFOSR FA9550-16-1-0102, NSF DMS-1620262, DMS-1720442, and NSFC No. 11688101.}}
\author{Jie Shen\footnote{Department of Mathematics, Purdue University, USA, Email: shen7@purdue.edu}\ \ and Jie Xu\footnote{LSEC \& NCMIS, Institute of Computational Mathematics and Scientific/Engineering Computing (ICMSEC), Academy of Mathematics and Systems Science (AMSS), Chinese Academy of Sciences, Beijing 100190, China. Email: xujie@lsec.cc.ac.cn}}
\date{}
\maketitle
\begin{abstract}
  We develop a set of numerical schemes for the Poisson--Nernst--Planck equations. We prove that our schemes are mass conservative, uniquely solvable and keep positivity unconditionally. Furthermore, the first-order scheme is proven to be unconditionally energy dissipative. These properties hold for various spatial discretizations. Numerical results are presented to validate these properties.  Moreover,  numerical results indicate that the second-order scheme is also energy dissipative, and  both the first- and second-order schemes preserve the maximum principle for cases where the equation satisfies the maximum principle. 
\end{abstract}


\section{Introduction}
The Poisson--Nernst--Planck (PNP) equations describe the dynamics of charged particles in the electric field that is also affected by these particles, and have been 
used to model  physical systems involving motions of charged particles, including electrochemistry \cite{bazant2004diffuse}, semiconductor \cite{gajewski1986basic,ringhofer1990semiconductor}, and several biological phenomena \cite{corrias2004global,biler1994debye,eisenberg1998ionic}. 
When discussing the interplay of electric field and flow field, the PNP equation can also be coupled with the Navier--Stokes equation \cite{schmuck2009analysis}. 

A distinct feature of the PNP equations is that they are built as Wasserstein gradient flows \cite{MR2401600}. Wasserstein gradient flows are usually used to describe evolution of the concentration $c$ which remains to be positive, given a positive initial condition. The dissipation operator in Wasserstein flow is nonlinear, given by $\nabla\cdot(c\nabla(\cdot))$, whose negativity also requires $c$ to be positive. Meanwhile, in many cases the energy is well-defined with a lower bound only when $c$ is positive (see for example the Doi--Onsager type equation for liquid crystals \cite{doi1988theory,xu2018onsager}). Numerically, it is thus crucial to construct schemes that preserve positivity. 

There are several techniques of designing energy dissipative time-discretized schemes for gradient flows, 
including convex splitting \cite{MR1249036,eyre1998unconditionally,shen2012second}, stabilization \cite{ZCST99, shen2010numerical}, auxiliary variable approaches \cite{BGG11,Gui.T13} (including IEQ \cite{yang2016linear,zhao2016numerical} and SAV \cite{shen2018scalar,shen2017new,shen2018convergence}). 
However, none of these techniques guarantees positivity, a prerequisite of the energy dissipation. 
Hence, these techniques can not be easily applied to Wasserstein gradient flows. Note however that a positivity preserving scheme for a Cahn--Hilliard equation with Flory-Huggins energy potential, which is not a Wasserstein gradient flow,  was recently developed in \cite{chen2017positivity}. 

As for the PNP equations, some  schemes with different properties have been  constructed \cite{flavell2014conservative,nanninga2008computational,lopreore2008computational,cagni2007effects,gardner2011electrodiffusion,gardner2004electrodiffusion,eisenberg2011mathematical,horng2012pnp}. Rigorous  numerical analyses for a set of finite-element approximations have been carried out  in \cite{prohl2009convergent}. 
Many Wasserstein gradient flows include a common dissipative term $\nabla\cdot (D\nabla c)$. In the context of Wasserstein gradient flow, to derive the energy dissipation, it needs to be interpreted as $\nabla\cdot(Dc\nabla\log c)$ to be consistent with other terms that take the form $\nabla\cdot(Dc\nabla(\cdot))$. 
The existing schemes are all based on the first interpretation and utilize standard time discretization, including implicit Euler, Crank--Nicolson, or backward differentiation formulas. Some of them  preserve positivity or some form of energy (not the entropy form) dissipation under certain conditions, but usually not both. In \cite{prohl2009convergent}, a quite complicated entropy-based scheme with regularized free energy is constructed, and proven to preserve   energy dissipation (in the entropy form), and  satisfy `quasi-non-negativity' (numerical solution bounded from below by an  arbitrarily small negative number) which is made possible by regularizing the mobility and free energy  so that it is well-defined for negative concentration. 
However, regularization cannot preserve positivity in the strong sense. It shall become clear that in order to construct schemes which preserve positivity and energy dissipation, one should deal with $\nabla\cdot(Dc\nabla\log c)$ instead of $\nabla\cdot (D\nabla c)$.

In this paper, we shall construct   schemes  for PNP equations which are 
\begin{enumerate}[(i)]
\item  mass conservative,
\item uniquely solvable, 
\item positivity preserving, and 
\item unconditionally energy dissipative. 
\end{enumerate}
We discretize the PNP equations in the context of Wasserstein gradient flow, based on the form $\nabla\cdot(Dc\nabla\log c)$. 
The appearance of logarithmic function in the schemes is essential to guarantee the concentration, which is also part of the diffusion coefficient, to be positive. 
This is attained by treating the coefficient $c$ explicitly, and $\log c$ from the variational derivative of the energy implicitly. 
The key for achieving the nice properties stated above is  that the schemes can be interpreted as  minimization of a strictly convex functional, which implies the uniquely solvability, positivity and energy dissipation.

We start by constructing a  first-order time discretization scheme and show that it satisfies the four properties (i)-(iv), and we believe that it is the only scheme which is positivity preserving and 
unconditionally energy dissipative.
We then construct a second-order scheme, and show that it satisfies the properties (i)--(iii). 
We also discuss how to construct full  discretization schemes which can preserve the  properties of the time discretization schemes. Although at each time step, these schemes lead to a nonlinear system but 
since its unique solution  is the minimizer of a strictly convex functional, it can be solved efficiently by  Newton's iteration.
We provide ample numerical results to  show that both first- and second-order schemes satisfy the four properties. Moreover, in some special cases where the solution of the PNP equation satisfies maximum principle and  electrostatic energy dissipation,  both the first- and second-order schemes can also preserve the maximum principle and electrostatic energy dissipation. 

The rest of paper is organized as follows. 
In Section \ref{PDE}, we introduce the PNP equations and state some of their properties that we desire to keep in numerical solutions. 
Then, we construct  numerical schemes in Section \ref{scheme} and prove that they satisfy the four properties stated earlier. We start by writing down the semi-discrete-in-time scheme, followed by careful discretization in space so that the properties of time discretization schemes can be preserved in the full discretization. 
Numerical results are presented in Section \ref{results}. 
Concluding remarks are given in the last section. 

\section{PNP equations\label{PDE}}
We first introduce the PNP equations in a general form, and then pay particular attention to a popular two-component system because it possesses extra properties. 
\subsection{General form}
We consider a system with $N$ species of charged particles driven by Brownian motion and the electric field in an open bounded domain $\Omega\subset \mathbb{R}^d \;(d=1,2,3)$. The system is charged with a fixed constant density $\rho_0$. 
To write down the dimensionless PNP equations governing the motion of this system,  we introduce some dimensionless quantities (functions) below: 
\begin{itemize}
\item $c_i(\bm{x})$ is the density of the $i$-th species;
\item  $\phi(\bm{x})$ is the internal electric potential contributed by the charged particles; $\phi_e(\bm{x})$ is a given external electric potential; 
\item The chemical potential w.r.t. $c_i$ is $\mu_i=\log c_i+z_i(\phi+\phi_e)$;
\item The constants $z_i,D_i>0$ are the valence and the diffusion constant of the $i$-th species, 
and $\epsilon>0$ is the permittivity. 
 \end{itemize}
Then, the PNP equations are  written as 
\begin{align}
  \frac{\partial c_i}{\partial t}
 =&\nabla\cdot(D_ic_i\nabla\mu_i)
  =\nabla\cdot\Big(D_ic_i\nabla\big(\log c_i+z_i(\phi+\phi_e)\big)\Big)\quad \text{in}\;\Omega\quad (i=1,\ldots,N), \label{PnP1}
\end{align}
where the internal electric potential $\phi$ is determined by 
\begin{equation}
  -\nabla\cdot(\epsilon\nabla\phi)=\rho_0+\sum_{i=1}^Nz_ic_i \quad \text{in}\;\Omega. \label{PnP2}
\end{equation}
Noticing that $\nabla c=c\nabla \log c$, we  can rewrite \eqref{PnP1} as
\begin{align}
  \frac{\partial c_i}{\partial t}
 =\nabla\cdot\Big(D_i\big(\nabla c_i+z_ic_i\nabla(\phi+\phi_e)\big)\Big)\quad \text{in}\;\Omega\quad (i=1,\ldots,N), \label{PnP3}
\end{align}
which is in the form most often used in the literature.

The boundary conditions are imposed on $\mu_i$ and $\phi$. 
They can be either  periodic on both $\mu_i$ and $\phi$; 
or, be of Neumann type  on the flux to guarantee the mass conservation, 
$$
c_i\frac{\partial\mu_i}{\partial\bm{n}}=c_i\frac{\partial\big(\log c_i+z_i(\phi+\phi_e)\big)}{\partial\bm{n}}=\frac{\partial c_i}{\partial\bm{n}}+z_ic_i\frac{\partial(\phi+\phi_e)}{\partial\bm{n}}=0, 
$$
and either Dirichlet, Neumann, or Robin boundary conditions on $\phi$, 
$$
\phi=0; \quad \text{ or }\frac{\partial\phi}{\partial\bm{n}}=0; \quad \text{ or }\alpha\phi+\beta\frac{\partial\phi}{\partial\bm{n}}=0,\ \alpha,\beta>0. 
$$
If using periodic or Neumann boundary conditions on $\phi$, we require that 
$$
\rho_0+\sum z_i\bar{c}_i=0, \qquad
\int_\Omega\phi\,\md \bm{x}=0, 
$$
where $\bar{c}_i$ is the average density of the $i$-th species. 

\begin{rem} We only consider periodic or homogeneous boundary conditions above on $\phi$. For non-homogeneous boundary conditions such as $\phi|_{\partial\Omega}=g$, we can split $\phi$ as $\phi_1+\phi_2$, with 
\begin{align*}
  &-\nabla\cdot(\epsilon\nabla\phi_1)=\rho_0+\sum_{i=1}^Nz_ic_i, \quad \phi|_{\partial\Omega}=0, \\
  &-\nabla\cdot(\epsilon\nabla\phi_2)=0, \quad \phi|_{\partial\Omega}=g. 
\end{align*}
Note that $\phi_2$ does not depend on $c_i$. Thus, $\phi_2$ actually acts as an external potential and could be added to $\phi_e$. 
It is known that the profiles of $c_i$ can sensitively depend on the boundary conditions \cite{flavell2014conservative}. 
In the above formulation, it actually implies that $\phi_2$, which goes in to the external potential, greatly affects the profile. 
\end{rem}

The total free energy of the system is given by 
\begin{equation}
  E(\{c_i\},\phi)=\int_\Omega \sum_{i=1}^Nc_i(\log c_i-1) + \left(\rho_0+\sum_{i=1}^Nz_ic_i\right)(\frac{1}{2}\phi+\phi_e)\,\md \bm{x}. \label{totalenergy}
\end{equation}
Assuming $\|\phi_e\|_{L^\infty(\Omega)}\le A$, we derive that the total energy is bounded from below. 
Indeed, we have 
$$c_i\log c_i-c_i+z_i\phi_ec_i\ge c_i\log c_i-(|z_i|A+1)c_i,$$
which is bounded from below. 
For the term with the internal potential $\phi$, we derive by integration by parts that 
\begin{align*}
  \int \left(\rho_0+\sum_{i=1}^Nz_ic_i\right)\phi\,\md \bm{x} 
  =\left\{
  \begin{array}{l}
    \displaystyle\int \epsilon|\nabla\phi|^2\,\md \bm{x}, \quad \text{with periodic, Dirichlet or Neumann B.C.},\\
    \displaystyle\int \epsilon|\nabla\phi|^2\,\md \bm{x}+\int_{\partial\Omega} \epsilon\frac{\alpha}{\beta}|\phi|^2\,\md S, \quad \text{with Robin B.C.}. 
  \end{array}
  \right.
\end{align*}

The PNP equations \eqref{PnP1}-\eqref{PnP2}   satisfy several important properties: 
\begin{enumerate}
\item Mass conservation: Integrating \eqref{PnP1} over $\Omega$, we obtain immediately
$$
\int_\Omega c_i(\cdot,t)\md \bm{x}=\int_\Omega c_i(\cdot,0)\md \bm{x}. 
$$
\item Positivity: The well-posedness of \eqref{PnP1}-\eqref{PnP2} (cf. \cite{schmuck2009analysis}) ensures that, if $c_i(\cdot, 0)>0$, then we still have $c_i(\cdot,t)>0$ for any $t>0$. 

\item Energy dissipation: 
  \begin{equation}\label{diss}
    \frac{\md E}{\md t}=-\int\sum_{i=1}^N D_ic_i\left|\nabla\mu_i\right|^2 \md \bm{x}. 
  \end{equation}
To derive the above energy dissipation, we need to observe that $\mu_i=\displaystyle\frac{\delta E}{\delta c_i}$. 
Actually, the variation $\delta\phi$ satisfies 
$$
-\nabla\cdot\big(\epsilon\nabla(\delta\phi)\big)=\sum_{i=1}^Nz_i\delta c_i, 
$$
with the same boundary conditions as $\phi$. Regardless of the type of the boundary conditions, we have 
$$
\int-\delta\phi\nabla\cdot(\epsilon\nabla\phi)\,\md \bm{x}=\int-\nabla\cdot\big(\epsilon\nabla(\delta\phi)\big)\phi \,\md \bm{x}. 
$$
So we have 
\begin{align*}
  &\delta\int\left(\rho_0+\sum_{i=1}^Nz_ic_i\right)\phi\,\md \bm{x}
  =\int\phi\sum_{i=1}^Nz_i\delta c_i-\delta\phi\nabla\cdot(\epsilon\nabla\phi)\,\md \bm{x}\\
  &\qquad=\int\phi\sum_{i=1}^Nz_i\delta c_i-\nabla\cdot\big(\epsilon\nabla(\delta\phi)\big)\phi\,\md \bm{x}
  =2\int\phi\sum_{i=1}^Nz_i\delta c_i\,\md \bm{x}
\end{align*}
Therefore, by multiplying the equation \eqref{PnP1} with $\mu_i$, taking the integral and summing up over $i$, we obtain \eqref{diss}.

\end{enumerate}

\subsection{A two-component system\label{twocomponent}}
We consider a two-component system ($N=2$) which has attracted special attention in the literature. 
We set $z_1=1$, $z_2=-1$, $\epsilon=1$ and the external electric potential $\phi_e=0$. 
Denote $p=c_1$ and $n=c_2$. Let the average density be $\bar{c}_1=\bar{c}_2$ so that  $\rho_0=0$. 
The PNP equation becomes 
\begin{align}
 & \frac{\partial p}{\partial t}
  =\nabla\cdot\Big(D_1p\nabla(\log p+\phi)\Big), \label{PnP2a} \\
&  \frac{\partial n}{\partial t}
  =\nabla\cdot\Big(D_2n\nabla(\log n-\phi)\Big),  \label{PnP2b} \\
 & -\Delta\phi=p-n,  \label{PnP2c}
\end{align}
where $p$ and $n$ denote the concentration of positively and negatively charged particles, respectively, and $\phi$ is the electronic potential. This system by W. Nernst and M. Planck to describe the potential difference in a galvanic cell (e.g., rechargeable batteries, or biological cells), and  has applications in many different fields including  chemistry,  biology,   plasma physics, and modeling of semi-conductor devices. 

The above system has two special properties stated below, which are satisfied only under the periodic or Neumann boundary conditions for $p,\ n,\ \phi$. 
They do not necessarily hold for the general form of PNP equations. 
\begin{enumerate}
\item  The electrostatic energy $\|\nabla \phi\|^2/2$ is dissipative if
$D_1=D_2=D$. Indeed,  multiplying \eqref{PnP2a} and   \eqref{PnP2b} with $\phi$ and calculating their difference, we obtain 
\begin{equation}
  \frac{\md \|\nabla\phi\|^2/2}{\md t}=
  -D\int \left[(p-n)^2+(p+n)|\nabla\phi|^2\right]\md \bm{x}.  
\end{equation}
It needs to be pointed out that the electrostatic energy is part of the total energy \eqref{totalenergy}, by noticing the derivation below \eqref{totalenergy}. 

\item 
The solutions $p$ and $n$ satisfy maximum principle which can be proved as follows. Multiplying \eqref{PnP2a} with $p^{k-1}$, we obtain
\begin{align*}
  \int\frac{1}{k}\cdot\frac{\partial p^k}{\partial t}\md \bm{x}
  =&-D_1\int (k-1)p^{k-2}|\nabla p|^2+p\nabla(p^{k-1})\cdot\nabla\phi \md \bm{x}, \\
  =&-D_1\int (k-1)p^{k-2}|\nabla p|^2+\frac{k-1}{k}\nabla(p^k)\cdot\nabla\phi \md \bm{x}, \\
  =&-D_1\int (k-1)p^{k-2}|\nabla p|^2-\frac{k-1}{k}p^k\Delta\phi \md \bm{x}, \\
  =&-D_1\int (k-1)p^{k-2}|\nabla p|^2+\frac{k-1}{k}p^k(p-n) \md \bm{x}. 
\end{align*}
Similarly, multiplying \eqref{PnP2b} with $n^{k-1}$,  we have 
\begin{align*}
  \int\frac{1}{k}\cdot\frac{\partial n^k}{\partial t}\md \bm{x}
  =&-D_2\int (k-1)n^{k-2}|\nabla n|^2-\frac{k-1}{k}n^k(p-n) \md \bm{x}. 
\end{align*}
Taking the sum of the above two equations, and noting that $p,\,n>0$, we obtain 
\begin{align}
  \frac{\partial (p^k/D_1+n^k/D_2)}{\partial t}
  =&-\int k(k-1)(p^{k-2}|\nabla p|^2+n^{k-2}|\nabla n|^2)\nonumber\\
  &\qquad +(k-1)(p^k-n^k)(p-n) \md \bm{x}\le 0. \label{L_k}
\end{align}
So we have 
\begin{align*}
\|p(t)\|_{L^k}\le &\left(\|p(t)\|_{L^k}^k+\frac{D_1}{D_2}\|n(t)\|_{L^k}^k\right)^{1/k}
\le \left(\|p(0)\|_{L^k}^k+\frac{D_1}{D_2}\|n(0)\|_{L^k}^k\right)^{1/k}\\
\le &\left(1+\frac{D_1}{D_2}\right)^{1/k}\max\{\|p(0)\|_{L^{\infty}},\|n(0)\|_{L^{\infty}}\}. 
\end{align*}
Taking the limit $k\to +\infty$, we obtain
$$
\max\{\|p(t)\|_{L^{\infty}},\|n(t)\|_{L^{\infty}}\}\le \max\{\|p(0)\|_{L^{\infty}},\|n(0)\|_{L^{\infty}}\}. 
$$
Note that the inequality \eqref{L_k} also holds for $k<-1$, we then obtain by taking $k\to -\infty$   that
$$
\max\left\{\left\|\frac{1}{p(t)}\right\|_{L^{\infty}},\left\|\frac{1}{n(t)}\right\|_{L^{\infty}}\right\}\le \max\left\{\left\|\frac{1}{p(0)}\right\|_{L^{\infty}},\left\|\frac{1}{n(0)}\right\|_{L^{\infty}}\right\}. 
$$

\end{enumerate}
Although we are not aiming to design numerical schemes guarateeing these two properties theoretically, we are still interested in and will examine whether they can be kept in the numerical solutions. 

\section{Numerical scheme\label{scheme}}
We construct in this section numerical schemes for \eqref{PnP1}-\eqref{PnP2}. 
We start from a first-order  scheme and prove that it enjoys the four nice properties described in the introduction. 
We then construct a second-order  scheme which enjoys the first three properties.

\subsection{First-order  scheme}
We first write down the time-discretized scheme for \eqref{PnP1}-\eqref{PnP2}:
\begin{align}
  &\frac{c_i^{n+1}-c_i^n}{\delta t}=\nabla\cdot(D_ic_i^n\nabla\mu_i^{n+1})
  =\nabla\cdot\Big(D_ic_i^n\nabla\big(\log c_i^{n+1}+z_i(\phi^{n+1}+\phi_e)\big)\Big),\quad i=1,\ldots,N, \label{first_1}\\
  &-\nabla\cdot(\epsilon\nabla\phi^{n+1})=\rho_0+\sum_{i=1}^Nz_ic_i^{n+1}, \label{first_2}
\end{align}
with the boundary conditions imposed on $\mu_i^{n+1}=\log c_i^{n+1}+z_i(\phi^{n+1}+\phi_e)$ and $\phi^{n+1}$ as in the PDE system \eqref{PnP1}-\eqref{PnP2}. 

\begin{theorem}\label{semidis}
  Assume $c_i^n>0$ for all $i$. For any solution to the scheme \eqref{first_1}-\eqref{first_2}, we have 
  \begin{enumerate}
  \item Mass conservation: 
    $$
    \int c_i^{n+1}\md \bm{x}=\int c_i^n\md \bm{x}. 
    $$
  \item Positivity preserving: $c_i^{n+1}>0$ for all  $i$. 
  \item Energy dissipation: 
    \begin{equation}\label{dissn}
    {E^{n+1}-E^n}\le -\delta t\int\sum_{i=1}^ND_ic_i^n\left|\nabla\mu_i^{n+1}\right|^2 \md \bm{x}, \quad n\ge 0. 
    \end{equation}
    where $E^k=\int_\Omega \sum_{i=1}^Nc^k_i(\log c^k_i-1) + \left(\rho_0+\sum_{i=1}^Nz_ic^k_i\right)(\frac{1}{2}\phi^k+\phi_e)\,\md \bm{x}$  \end{enumerate}
\end{theorem}
\begin{proof}
  We shall only prove the theorem for the Neumann boundary conditions on $\phi$ and $\mu_i$. The results with other boundary conditions can be proved in the same way, as we will point out afterwards. 

  Taking the integral of \eqref{first_1} over $\Omega$ and using the Neumann boundary conditions on the chemical potential, we obtain the mass conservation. 
 
  The positivity follows from the appearance of $\log c_i^{n+1}$. 
  
  It remains to prove the energy dissipation. To this end, we take  the inner product of the equation  \eqref{first_1} with $\log c_i^{n+1}+z_i(\phi^{n+1}+\phi_e)$,  summing up for $1\le i\le N$, we arrive at 
  \begin{align*}
    &\sum_{i=1}^N(c_i^{n+1}-c_i^n,\log c_i^{n+1})+\left(\rho_0+\sum_{i=1}^Nz_ic_i^{n+1}-\rho_0-\sum_{i=1}^Nz_ic_i^n,\phi^{n+1}+\phi_e\right)\\
    =&\sum_{i=1}^N(c_i^{n+1}-c_i^n,\log c_i^{n+1})+\left(\nabla\phi^{n+1}-\nabla\phi^n,\epsilon\nabla\phi^{n+1}\right)+\left(\rho_0+\sum_{i=1}^Nz_ic_i^{n+1}-\rho_0-\sum_{i=1}^Nz_ic_i^n,\phi_e\right)\\
    =&-\delta t\int\sum_{i=1}^ND_ic_i^n\left|\nabla(\log c_i^{n+1}+z_i(\phi^{n+1}+\phi_e))\right|^2 \md \bm{x}. 
  \end{align*}
  We note  that by Taylor expansion we have 
  \begin{equation}\label{logdiss}
  (a-b)\log a=(a\log a-a)-(b\log b-b)+\frac{(a-b)^2}{2\xi}, \quad \xi\in [\min\{a,b\},\max\{a,b\}]. 
  \end{equation}
  We also have  
  \begin{equation}\label{iden2}
  (\nabla\phi^{n+1}-\nabla\phi^n)\cdot\nabla\phi^{n+1}=\frac{1}{2}(|\nabla\phi^{n+1}|^2-|\nabla\phi^n|^2+|\nabla\phi^{n+1}-\nabla\phi^{n+1}|^2). 
   \end{equation}
  With the above equalities, we immediately derive \eqref{dissn}. 
\end{proof}

It remains to examine whether there exists a solution for the scheme. 
Below, we give a formal derivation by formulating it as the minimizer of a strictly convex functional. 
Still, we examine the Neumann boundary conditions for $\phi$ and $\mu_i$. 
Let us introduce  linear operators $\mathcal{L}_i^n$, which are defined as follows: let $\mathcal{L}_i^ng=u$ if they satisfy the following elliptic equation with the Neumann boundary conditions, 
\begin{align*}
  -\nabla\cdot(c_i^n\nabla u)=g,\quad 
  \int u \md \bm{x}=0. 
\end{align*}
Also, we define $\mathcal{L}$ as above where we replace $c_i^n$ with $\epsilon$. 
The linear operators $\mathcal{L}_i^n$ and $\mathcal{L}$ are symmetric and nonnegative in the sense $(u,\mathcal{L}u)\ge 0$. 
We consider the following functional 
\begin{align}
  F[c_i^{n+1}]=&
  \sum_{i=1}^N(c_i^{n+1}\log c_i^{n+1}-c_i^{n+1},1)
  +\frac{1}{2\delta t}\sum_{i=1}^N\Big(c_i^{n+1}-c_i^n,\mathcal{L}_i^n(c_i^{n+1}-c_i^n)\Big)\nonumber\\
  &+\frac{1}{2}\left(\rho_0+\sum_{i=1}^Nz_ic_i^{n+1},\mathcal{L}\Big(\rho_0+\sum_{i=1}^Nz_ic_i^{n+1}\Big)\right)
  +\left(\rho_0+\sum_{i=1}^Nz_ic_i^{n+1},\phi_e\right). \label{fnplus1}
\end{align}
The above functional is strictly convex, because $\displaystyle\int c_i^{n+1}\log c_i^{n+1}+(z_i\phi_e-1)c_i^{n+1}\md \bm{x}$ is strictly convex about $c_i^{n+1}$, and the remaining terms give a quadratic nonnegative functional. 
Its Euler-Langrange equation under the constraints of mass is 
\begin{align*}
  &\frac{1}{\delta t}\mathcal{L}_i^n(c_i^{n+1}-c_i^n)+\log c_i^{n+1}+z_i\mathcal{L}\Big(\rho_0+\sum_{i=1}^Nz_ic_i^{n+1}\Big)+z_i\phi_e\\
  &\qquad =\frac{1}{\delta t}\mathcal{L}_i^n(c_i^{n+1}-c_i^n)+\log c_i^{n+1}+z_i(\phi^{n+1}+\phi_e)=\lambda_i, \\
  &\int(c_i^{n+1}-c_i^n)\md\bm{x}=0. 
\end{align*}
where $\lambda_i$ are the Lagrange multipliers for the mass conservation. 
It is easy to see that the above equations are equivalent to \eqref{first_1}--\eqref{first_2}. 
The functional $F$ has a unique minimizer. 
Moreover, the minimizer cannot have $c_i^{n+1}(\bm{x})=0$, because the derivative of the term $c_i^{n+1}\log c_i^{n+1}-c_i^{n+1}$ has the derivative $\log c_i^{n+1}$ that tends to $-\infty$ at zero. 
Hence, the unique minimizer must have $c_i^{n+1}>0$ for all $i$, which is the unique solution to the Euler-Lagrange equation, hence to the scheme, because we can solve $\phi$ uniquely from \eqref{first_2}.

The above formal derivation can be converted into a rigorous proof after we discretize in space. Before going on, let us explain the difference when using other boundary conditions, both for the theorem and for the formal derivation above. For the periodic boundary conditions, everything is exactly the same. 
When using Dirichlet or Robin boundary conditions, we do not need the average equals to zero when defining the operator $\mathcal{L}$ (but still need for $\mathcal{L}^n$). 
For the energy dissipation for Robin boundary conditions, we need an extra term $\displaystyle \int_{\partial\Omega}\epsilon\alpha\beta|\phi|^2\md S$, which can be dealt with in the same way as $\displaystyle \int_{\Omega}\epsilon|\nabla\phi|^2\md \bm{x}$. 
Thus, we will still focus on the Neumann boundary conditions below.

We now discuss how to construct spatial discretizations which preserve the nice properties for  the  scheme  \eqref{first_1}-\eqref{first_2}.
Note that in the proof of {\bf Theorem 3.1}, we have used non-standard functions like $\log c_i^{n+1}$ as test function. Therefore, the proof can not be directly extended to a straightforward discretization in space since the discrete version of $\log c_{i}^{n+1}$ is usually not in the discrete test space.
We need to carefully discretize the space to keep the properties stated in Theorem \ref{semidis} in the discrete sense. 

Let us first discuss Galerkin type discretizations with finite-elements or spectral methods. 
Since there are differential operators with variable coefficients, we need to define a discrete inner product, i.e. numerical integration, on a finite set of points $Z=\{\bm{z}\}$: 
\begin{equation}\label{numint}
  [u,v]=\sum_{\bm{z}\in Z} \beta_{\bm{z}}u(\bm{z})v(\bm{z}), 
\end{equation}
where we require that the weights $\beta_{\bm{z}}>0$. 
For finite element methods, the sum should be understood as  $\sum_{K\subset \mathcal{T}}\sum_{\bm{z}\in Z(K)}$  where $\mathcal{T}$ is a given triangulation. 

As we have mentioned, we still consider Neumann boundary conditions. Let $X_M\subset H^1(\Omega)$ be a finite dimensional approximation space. 
Assume that there is a unique function $\psi_{\bm{z}}(\bm {x})$ in $X_M$ satisfying $\psi_{\bm{z}}(\bm{z'})=\delta_{\bm{z}\bm{z'}}$ for $\bm{z},\bm{z'}\in Z$. 
Then, we can define $I_M: C(\Omega) \rightarrow X_M$ as the interpolation operator about the points in $Z$. 

Our Galerkin method  for the first-order scheme \eqref{first_1}-\eqref{first_2} is: to find $\{c_i^{n+1}\}$ and $\phi^{n+1}$ in $X_M$ satisfying 
\begin{align}
  &\Big[\frac{c_i^{n+1}-c_i^n}{\delta t},v\Big]
  =-\Big[D_ic_i^n\nabla\Big(I_M\big(\log c_i^{n+1}+z_i(\phi^{n+1}+\phi_e)\big)\Big),\nabla v\Big],\quad v\in X_M, \label{Glkc}\\
  &(\epsilon\nabla\phi^{n+1},\nabla w)=\Big[\rho_0+\sum_{i=1}^Nz_ic_i^{n+1},w\Big],\quad w\in X_M. \label{Glkphi}
\end{align}
We emphasize that in the above, $(\cdot,\cdot)$ represents the continuous $L^2$ inner product, while $[\cdot,\cdot]$ represents the discrete $L^2$ inner product defined in \eqref{numint}.
\begin{theorem}\label{Glkdis}
The fully discretized scheme \eqref{Glkc}-\eqref{Glkphi} enjoys the following properties:
  \begin{enumerate}
  \item Mass conservation: 
    $$
    [c_i^{n+1},1]=[c_i^n,1]. 
    $$
  \item Unique solvability:  the scheme  \eqref{Glkc}-\eqref{Glkphi} possesses a unique solution $(\{c_i^{n+1}\in X_M\},\phi^{n+1}\in X_M)$. 
   \item Positivity preserving: if $c_i^n(\bm{z})>0$ for all $i$ and $\bm{z}\in Z$, we have $c_i^{n+1}(\bm{z})>0$ for all $i$ and $\bm{z}\in Z$. 
  \item Energy dissipation:    
    \begin{equation}\label{dis_diss}
    {\tilde{E}^{n+1}-\tilde{E}^n}\le -\delta t\sum_{i=1}^N[D_ic_i^n\nabla\mu_i^{n+1},\nabla\mu_i^{n+1}], 
    \end{equation} 
    where $\mu_i^{n+1}=I_M\big(\log c_i^{n+1}+z_i(\phi^{n+1}+\phi_e)\big)$ and the discrete energy is defined as
    \begin{equation}
      \tilde{E}^n=\sum_{i=1}^N[c_i^{n}\log c_i^{n}-c_i^{n},1]+\left[\rho_0+\sum_{i=1}^Nz_ic_i^{n},\frac{1}{2}\phi^{n}+\phi_e\right]. 
    \end{equation}
  \end{enumerate}
\end{theorem}
\begin{proof}
  The mass conservation is obtained by choosing $v=1$. 

  Next, we look the  unique solvability and positivity. 
  Since we have $c_i^n(\bm {x})=\sum_{\bm{z}}c_i^n(\bm{z})\psi_{\bm{z}}(\bm {x})$, 
  let us denote the vector $(c_i^n(\bm{z}), \;\bm{z}\in Z)$ as $\tilde{c}_i^n$. 
  Similarly we denote  $({\phi}^{n}(\bm{z}),\bm{z}\in Z)$ and
  $({\phi}_e^{n}(\bm{z}),\bm{z}\in Z)$  by the vectors $\tilde{\phi}^{n}$ and $\tilde{\phi}_e$, respectively. 
  We define the following stiffness and mass matrices: 
  \begin{align*}
   A_i^n=[D_ic_i^n\nabla\psi_{\bm{z}},\nabla\psi_{\bm{z'}}], \quad 
    A=\epsilon(\nabla\psi_{\bm{z}},\nabla\psi_{\bm{z'}}),\quad 
    B=[\psi_{\bm{z}},\psi_{\bm{z'}}].
      \end{align*}
It is clear that $B$ is a  diagonal matrix with positive elements, $A$ is symmetric positive semi-definite. If $c_i^n(\bm{z})>0$ for $\bm{z}\in Z$, the matrices $A_i^n$   are symmetric positive semi-definite. 
Furthermore, $A_i^n \tilde x=0$, similarly $A\tilde x=0$, if and only if all the components of $\tilde x$ are equal. Therefore,  $A_i^n$ and $A$ have one zero eigenvalue with all other eigenvalues being positive. Hence,   
  the eigen-decomposition of $A$ takes the form $A=T^t\Lambda T$ with $\Lambda=\mbox{diag}(0,\lambda_2,\cdots,\lambda_M)$ and $\lambda_j>0$ for $j=2,\cdots,M$. We denote by $A^*$ the pseudo-inverse given by $A^*=T^t\mbox{diag}(0,\lambda_2^{-1},\cdots,\lambda_M^{-1})T$. Similarly we can define $(A_i^n)^*$ for $i=1,\cdots,N$. 
With the above notations, we can rewrite the scheme \eqref{Glkc}-\eqref{Glkphi} in matrix form as follows: 
  \begin{align}
    &\frac{1}{\delta t}B(\tilde{c}_i^{n+1}-\tilde{c}_i^{n})=-A_i^n\big(\log \tilde{c}_i^{n+1}+z_i(\tilde{\phi}^{n+1}+\tilde{\phi}_e)\big), \\
    &A\tilde{\phi}^{n+1}=B\left(\rho_0+\sum_{i=1}^Nz_i\tilde{c}_i^{n+1}\right). \label{tphi0}
  \end{align}
  
  Multiplying the above equations by pseudo-inverse $(A_i^n)^*$ and $A^*$, we find
  \begin{align}
    &\frac{1}{\delta t}(A_i^n)^*B(\tilde{c}_i^{n+1}-\tilde{c}_i^{n})+\log \tilde{c}_i^{n+1}+z_i(\tilde{\phi}^{n+1}+\tilde{\phi}_e)=\lambda_i\bm{1}, \\
    &\tilde{\phi}^{n+1}=A^*B(\rho_0+\sum_{i=1}^Nz_i\tilde{c}_i^{n+1})+\lambda\bm{1}, \label{tphi}
  \end{align}
  with $\bm{1}$ representing the all-one vector. 
  Eliminating $\tilde{\phi}^{n+1}$ from the above, and  then multiplying  $B$ to the first equation, we arrive at 
  \begin{align*}
    &\frac{1}{\delta t}B(A_i^n)^*B(\tilde{c}_i^{n+1}-\tilde{c}_i^{n})+B\log \tilde{c}_i^{n+1}+z_i\Big(BA^*B(\rho_0+\sum_{i=1}^Nz_i\tilde{c}_i^{n+1})+B\tilde{\phi}_e\Big)=\lambda_i'B\bm{1}, 
  \end{align*}
  along with the mass conservation $\bm{1}^tB(\tilde{c}_i^{n+1}-\tilde{c}_i^n)=0$. 
  One can then easily check that the above  is the Euler-Lagrange equation of the function
  \begin{align*}
    \tilde{F}[\tilde{c}_i^{n+1}]=&\frac{1}{2\delta t}\sum_{i=1}^N(\tilde{c}_i^{n+1}-\tilde{c}_i^{n})^tB(A_i^n)^*B(\tilde{c}_i^{n+1}-\tilde{c}_i^{n})
    +\sum_{i=1}^N(\tilde{c}_i^{n+1})^tB(\log \tilde{c}_i^{n+1}-1)\\
    &+\frac{1}{2}\left(\rho_0+\sum_{i=1}^Nz_i\tilde{c}_i^{n+1}\right)^tBA^*B\left(\rho_0+\sum_{i=1}^Nz_i\tilde{c}_i^{n+1}\right)+\tilde{\phi}_e^tB\left(\rho_0+\sum_{i=1}^Nz_i\tilde{c}_i^{n+1}\right). 
  \end{align*}
 Since $B$ is diagonal and positive definite, and $(A_i^n)^*,\,A^*$ are symmetric and nonnegative,  it is clear that the above function is strictly convex about $\tilde{c}_i^{n+1}$.
 Therefore, $\tilde{F}[\tilde{c}_i^{n+1}]$ has a unique minimizer. 
 Below we eliminate the possibility of $\tilde{c}_i^{n+1}(\bm{z})=0$. 
 If this is done, the unique minimizer satisfies $\{\tilde{c}_i^{n+1}>0\}_{i=1,\cdots,N}$. 
 With $\tilde{c}_i^{n+1}$, we can then determine a unique $\tilde \phi^{n+1}$ from \eqref{tphi0}. 

 Let us prove by contradiction. Without loss of generality, suppose the minimizer has $\tilde{c}_1^{n+1}(\bm{z})=0$. Choose another $\bm{z'}$ such that $\tilde{c}_1^{n+1}(\bm{z'})>0$. Keep the other $\tilde{c}_i^{n+1}$, and substitute $\tilde{c}_1^{n+1}$ by $\tilde{d}_1^{n+1}=\tilde{c}_1+\beta_{\bm{z'}}\rho\bm{e}_{\bm{z}}-\beta_{\bm{z}}\rho\bm{e}_{\bm{z'}}$, where we use $\bm{e}_{\bm{z}}$ to denote the vector with the entry one for the $\bm{z}$-component and zero entry for others. 
Next, we will show that when $\rho$ is small enough, $\tilde{F}[\tilde{d}_1^{n+1},\tilde{c}_i^{n+1}|_{i=2,\ldots,n}]<\tilde{F}[\tilde{c}_i^{n+1}]$. 
In the following, we denote two quantities in the inequality in short by $\tilde{F}[\tilde{d}_1^{n+1}]$ and $\tilde{F}[\tilde{c}_1^{n+1}]$. 

Split $\tilde{F}$ into two parts: 
$$
\tilde{F}_1=\sum_{i=1}^N(\tilde{c}_i^{n+1})^tB(\log \tilde{c}_i^{n+1}-1), 
$$
and $\tilde{F}_2=\tilde{F}-\tilde{F}_1$. 
Note that $\tilde{F}_2$ is a quadratic function. Thus, there exists a constant $A_1>0$ such that for $\rho$ small enough, 
$$
|\tilde{F}_2[\tilde{d}_1^{n+1}]-\tilde{F}_2[\tilde{c}_1^{n+1}]|<A_1\rho. 
$$
Now we turn to $\tilde{F}_1$. Let $a=\tilde{c}_1^{n+1}(\bm{z'})>0$. We can calculate that 
$$
\tilde{F}_1[\tilde{d}_1^{n+1}]-\tilde{F}_1[\tilde{c}_1^{n+1}]=
\beta_{\bm{z}}\beta_{\bm{z'}}\rho\log(\beta_{\bm{z'}}\rho)+\beta_{\bm{z'}}\big((a-\beta_{\bm{z}}\rho)\log(a-\beta_{\bm{z}}\rho)-a\log a\big). 
$$
Since $a>0$, for $\rho$ small enough, we have 
$$
|(a-\beta_{\bm{z}}\rho)\log(a-\beta_{\bm{z}}\rho)-a\log a|<A_2\rho. 
$$
Thus, if we choose $\beta_{\bm{z}}\beta_{\bm{z'}}\log(\beta_{\bm{z'}}\rho)<-A_1-\beta_{\bm{z'}}A_2$, we arrive at $\tilde{F}[\tilde{d}_1^{n+1}]<\tilde{F}[\tilde{c}_1^{n+1}]$, which is the contradiction we want. 

It remains to prove  the energy dissipation. To this end, we choose $v=\delta t\mu_i^{n+1}=\delta tI_M\big(\log c_i^{n+1}+z_i(\phi^{n+1}+\phi_e)\big)$ in \eqref{Glkc} and take the sum for $1\le i\le N$, leading to 
  \begin{align*}
    &-\delta t\sum_{i=1}^N[D_ic_i^n\nabla\mu_i^{n+1},\nabla\mu_i^{n+1}]\\
    =&\sum_{i=1}^N\big[c_i^{n+1}-c_i^n,I_M\big(\log c_i^{n+1}+z_i(\phi^{n+1}+\phi_e)\big)\big]\\
    =&\sum_{i=1}^N[c_i^{n+1}-c_i^n,\log c_i^{n+1}]+\left[\rho_0+\sum_{i=1}^Nz_ic_i^{n+1}-\rho_0-\sum_{i=1}^Nz_ic_i^n,\phi^{n+1}+\phi_e\right].
  \end{align*}
  Then, by using \eqref{Glkphi} and \eqref{iden2}, we have
  \begin{align*}
    &2\left[\rho_0+\sum_{i=1}^Nz_ic_i^{n+1}-\rho_0-\sum_{i=1}^Nz_ic_i^n,\phi^{n+1}\right]\\
    =&2(\nabla\phi^{n+1}-\nabla\phi^n,\epsilon\nabla\phi^{n+1}) \\
    =&\Big((\nabla\phi^{n+1},\epsilon\nabla\phi^{n+1})-(\nabla\phi^{n},\epsilon\nabla\phi^n)+\big(\nabla(\phi^{n+1}-\phi^n),\epsilon\nabla(\phi^{n+1}-\phi^n)\big)\Big)\\
    =&\left[\rho_0+\sum_{i=1}^Nz_ic_i^{n+1},\phi^{n+1}\right]-\left[\rho_0+\sum_{i=1}^Nz_ic_i^{n},\phi^{n}\right]+\big(\nabla(\phi^{n+1}-\phi^n),\epsilon\nabla(\phi^{n+1}-\phi^n)\big). 
  \end{align*}
 We can then obtain \eqref{dis_diss} by using \eqref{logdiss}. 
\end{proof}
\begin{rem}
  For Dirichlet boundary conditions on $\phi$, we just need to change the function space for $\phi$ and $w$ from $X_M$ to $X_{M0}$ requiring that the boundary value is zero. For Robin boundary conditions on $\phi$, we just need to add the surface integral in \eqref{Glkphi}. 
\end{rem}

Let us now briefly discuss how to construct finite difference schemes which preserve the properties of the time discretizations in the last section. 
An important aspect in finite difference schemes is to carefully implement the boundary conditions such that the summation by parts holds, which is crucial to guarantee the mass conservation (cf. \cite{flavell2014conservative} for comparison of non-conservative vs conservative discretization) and to derive the energy dissipation. This is not difficult on rectangular domains. We write down the 2D case, which is to be used in our numerical test, with the domain $[0,L]^2$ discretized at $M^2$ points $x_{j,k}=\Big((j-\frac{1}{2})\delta x,(k-\frac{1}{2})\delta x\Big),\; j,k=1,\cdots,M$ where $\delta x=L/M$. The scheme is written as 
\begin{align}
 &\frac{(c_i)_{j,k}^{n+1}-(c_i)_{j,k}^n}{\delta t}
  =\frac{D_i}{\delta x^2}\Big[\frac{(c_i)_{j+1,k}^{n}+(c_i)_{j,k}^{n}}{2}\Big((\mu_i)_{j+1,k}^{n+1}-(\mu_i)_{j,k}^{n+1}\Big)\label{FD1}\\
    &\qquad-\frac{(c_i)_{j,k}^n+(c_i)_{j-1,k}^n}{2}\Big((\mu_i)_{j,k}^{n+1}-(\mu_i)_{j-1,k}^{n+1}\Big),\nonumber\\
    &\qquad+\frac{(c_i)_{j,k+1}^{n}+(c_i)_{j,k}^{n}}{2}\Big((\mu_i)_{j,k+1}^{n+1}-(\mu_i)_{j,k}^{n+1}\Big)\nonumber\\
    &\qquad-\frac{(c_i)_{j,k}^n+(c_i)_{j,k-1}^n}{2}\Big((\mu_i)_{j,k}^{n+1}-(\mu_i)_{j,k-1}^{n+1}\Big)\Big],\;1\le j,k\le M,\; 1\le i\le N, \nonumber\\
  &-\epsilon\frac{\phi_{j+1,k}^{n+1}+\phi_{j-1,k}^{n+1}+\phi_{j,k+1}^{n+1}+\phi_{j,k-1}^{n+1}-4\phi_{j}^{n+1}}{h^2}=\rho_0+\sum_{i=1}^Nz_i(c_i)_{j,k}^{n+1}, \;1\le j,k\le M, \label{FD2}
\end{align}
where $(\mu_i)_{j,k}^n=\big(\log c_i+z_i(\phi+\phi_e)\big)_{j,k}^n$. 
To fix the idea, we still consider the Neumann boundary conditions. To have the summation by parts, we shall impose boundary terms like below, 
\begin{equation}\label{FD3}
\begin{split}
&\frac{(\mu_i)_{0,k}^{n+1}-(\mu_i)_{1,k}^{n+1}}{h}=0,\quad \frac{(\mu_i)_{M+1,k}^{n+1}-(\mu_i)_{M,k}^{n+1}}{h}=0, \\
&\frac{\phi_{0,k}^{n+1}-\phi_{1,k}^{n+1}}{h}=0,\quad \frac{\phi_{M+1,k}^{n+1}-\phi_{M,k}^{n+1}}{h}=0.
\end{split}
\end{equation} 
The above boundary discretization is for $\partial u/\partial\bm{n}|_{\partial\Omega}$. The term $u|_{\partial\Omega}$ shall be discretized by $\frac{1}{2}(u_{0,k}+u_{1,k})$ for the summation by parts, if we consider Dirichlet or Robin boundary conditions on $\phi$. 

For the above scheme, we have
\begin{theorem}
The finite difference scheme \eqref{FD1}-\eqref{FD3} enjoys the following properties:
  \begin{enumerate}
  \item Mass conservation: 
    $$
    \delta x^2\sum_{j,k=1}^M(c_i)_{j,k}^{n+1}=\delta x^2\sum_{j,k=1}^M(c_i)_{j,k}^n,\; 1\le i\le N. 
    $$
  \item Unique solvability:  the scheme  \eqref{FD1}-\eqref{FD3} possesses a unique solution $(\{(c_i)_{j,k}^{n+1}\},\phi_{j,k}^{n+1})$. 
   \item Positivity preserving: if $(c_i)_{j,k}^n>0$ for all $i$ and $(j,k)$, we have $(c_i)_{j,k}^{n+1}>0$ for all $i$ and $(j,k)$. 
  \item Energy dissipation:    we have 
    \begin{align}
    {\tilde{E}^{n+1}-\tilde{E}^n}\le -\delta t\sum_{i=1}^N\frac{D_i}{\delta x^2}&\sum_{\substack{1\le j\le M-1\\1\le k\le M}}\frac{(c_i)_{j+1,k}^{n}+(c_i)_{j,k}^{n}}{2}\Big((\mu_i)_{j+1,k}^{n+1}-(\mu_i)_{j,k}^{n+1}\Big)^2\nonumber\\
    &+\sum_{\substack{1\le j\le M\\1\le k\le M-1}}\frac{(c_i)_{j,k+1}^{n}+(c_i)_{j,k}^{n}}{2}\Big((\mu_i)_{j,k+1}^{n+1}-(\mu_i)_{j,k}^{n+1}\Big)^2, 
    \end{align} 
    where the discrete energy is defined as 
    \begin{equation}
      \tilde{E}^n=\sum_{i=1}^N\sum_{j,k=1}^M(c_i)_{j,k}^{n}(\log (c_i)_{j,k}^{n}-1)+\sum_{j,k=1}^M\Big[\rho_0+\sum_{i=1}^Nz_i(c_i)_{j,k}^{n}\Big]\cdot\Big[\frac{1}{2}\phi_{j,k}^{n}+(\phi_e)_{j,k}\Big]. 
    \end{equation}
  \end{enumerate}
\end{theorem}
\begin{proof}
  The mass conservation is obtained by taking the sum over $1\le j,k\le M$ on \eqref{FD1} and using the boundary conditions of $(\mu_i)_{j,k}^{n+1}$ in \eqref{FD3}. 

  The unique solvability and positivity can be proved similar to Theorem \ref{Glkdis} by choosing the matrices as those given by finite difference discretization. 

  The energy dissipation is derived by multiplying \eqref{FD1} with $(\mu_i)_{j,k}^{n+1}$ and taking the sum over $1\le j,k\le M$. On the right-hand side, the summation by parts is then done by noting the boundary conditions of $(\mu_i)_{j,k}^{n+1}$. 
On the left-hand side, we deal with the terms with $\phi_{j,k}^{n+1}$ in the same way as the last equation in the proof of Theorem \ref{Glkdis}, using \eqref{FD2}. 
\end{proof}

\subsection{Second-order scheme}
Apparently we can use second-order BDF scheme with Adams-Bashforth extrapolation to construct a second-order scheme. However, since the Adams-Bashforth extrapolation can not preserve positivity, we need to  modify it with
\begin{equation}
  \bar{c}_i=\left\{
  \begin{array}{ll}
    2c_i^n-c_i^{n-1},& \text{if }c_i^n\ge c_i^{n-1}, \\
    \frac{1}{2/c_i^n-1/c_i^{n-1}},& \text{if }c_i^n<c_i^{n-1}. 
  \end{array}
  \right.
\end{equation}
Then, a second order fully-discretized scheme can be written as follows:
to find $\{c_i^{n+1}\}$ and $\phi^{n+1}$ in $X_M$ satisfying 
\begin{align}
  &\Big[\frac{3c_i^{n+1}-4c_i^n+c_i^{n-1}}{2\delta t},v\Big]
  =-\Big[D_i\bar{c}_i\nabla\Big(I_M\big(\log c_i^{n+1}+z_i(\phi^{n+1}+\phi_e)\big)\Big),\nabla v\Big],\quad v\in X_M, \label{Glkc2}\\
  &(\epsilon\nabla\phi^{n+1},\nabla w)=\Big[\rho_0+\sum_{i=1}^Nz_ic_i^{n+1},w\Big],\quad w\in X_M. \label{Glkphi2}
\end{align}
Similar to the first-order scheme, we have 
\begin{theorem}\label{Glkdis2}
The fully discretized scheme \eqref{Glkc2}-\eqref{Glkphi2} enjoys the following properties:
  \begin{enumerate}
  \item Mass conservation: 
    $$
    [c_i^{n+1},1]=[c_i^n,1]. 
    $$
  \item Unique solvability:  the scheme  \eqref{Glkc2}-\eqref{Glkphi2} possesses a unique solution $(\{c_i^{n+1}\in X_M\},\phi^{n+1}\in X_M)$. 
   \item Positivity preserving: $c_i^{n+1}(\bm{z})>0$ for all $i$ and $\bm{z}\in Z$. 
  \end{enumerate}
\end{theorem}
\begin{rem}
 Unfortunately, we are unable to prove the energy dissipation. The reason is that we do not have an analog of \eqref{logdiss} to deal with the term
$(3c_i^{n+1}-4c_i^n+c^{n-1}_i,\log c_i^{n+1})$.
\end{rem}

\section{Numerical experiments\label{results}}
In this section, we present several numerical experiments to validate our theoretical results in the previous section. 
We first present two examples to examine accuracy and stability of our schemes. 
In these two examples, the equations are solved in $[0,2\pi]^2$ with periodic boundary conditions and discretized by Fourier spectral method in space. 
We will verify the convergence order as well as the mass conservation, positivity preserving and energy dissipation. 
Then, we present two other examples with Dirichlet and Neumann boundary conditions, one for two species and one for three species, on the domain $[0,1]^2$, discretized with the finite difference scheme \eqref{FD1}-\eqref{FD2}. 

Note that at each time step, the scheme is nonlinear, but it is shown that it possesses a unique solution which is the minimizer of a strictly convex function. Hence, it can be solved efficiently by Newton's iteration method. 
For a given Newton's direction, line search is incorporated to obtain a damped step length. 
We adopt a simple backtracking line search method, to half the step length until the residue of the nonlinear equations decreases, which requires the concentration to be positive since we have logarithm functions in the nonlinear equations. 
The linear system to obtain the Newton's direction is solved using the preconditioned GMRES iteration. 
For Fourier spatial discretization, we utilize the preconditioner given by choosing $\{c_i\}, \phi$ as constant functions. 
For finite difference discretization, the preconditioner is constructed by incomplete LU factorization without filling. 
For both Newton's and GMRES iterations, the tolerance is chosen as $10^{-6}$. 
This approach proves to be quite efficient, as we will present below. 

\begin{figure}
  \centering
  \includegraphics[width=.6\textwidth,keepaspectratio]{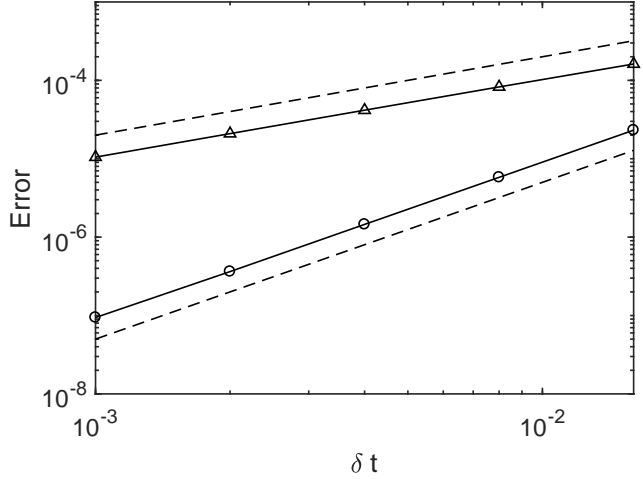}
  \caption{(Example 1) Convergence rate of two schemes (triangle: first order; circle: second order). The error is calculated as $\sqrt{\|p^n(\cdot)-p(\cdot,t^n)\|^2+\|n^n(\cdot)-n((\cdot,t^n)\|^2}$. The dashed lines represent the reference to first and second order convergence. }\label{cnvg}
\end{figure}
\textbf{Example 1} (Accuracy test). 
Let $z_1=1$, $z_2=-1$, $p=c_1$, $n=c_2$, $D_1=D_2=1$, and $\rho_0=0$, like in section \ref{twocomponent}. 
We set the external field $\phi_e=0$ and $\epsilon=1$. 
We use the first-order and second-order schemes with $64\times 64$ Fourier spectral modes for spatial discretization. 
The initial value is chosen as 
\begin{align*}
  p(x,0)=1.1+\sin x\cos y, \quad n(x,0)=1.1-\sin x\cos y. 
\end{align*}
The reference solution is obtained by the second order scheme with $\delta t=10^{-4}$. 
The errors by the two schemes are plotted in Fig.~\ref{cnvg}, which clearly shows the expected first and second order  accuracy. 
\begin{figure}
  \centering
  \includegraphics[width=.45\textwidth,keepaspectratio]{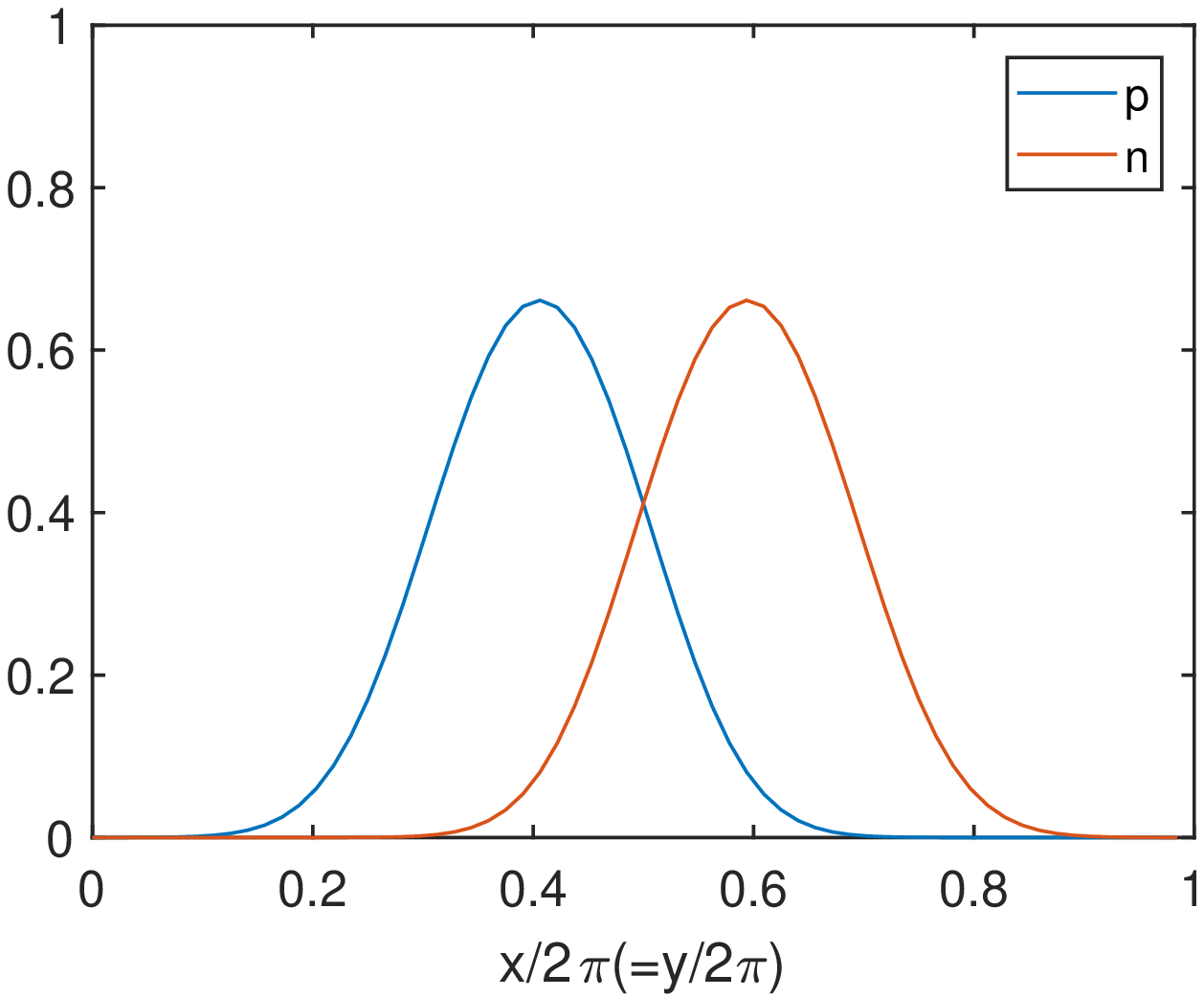}
  \includegraphics[width=.45\textwidth,keepaspectratio]{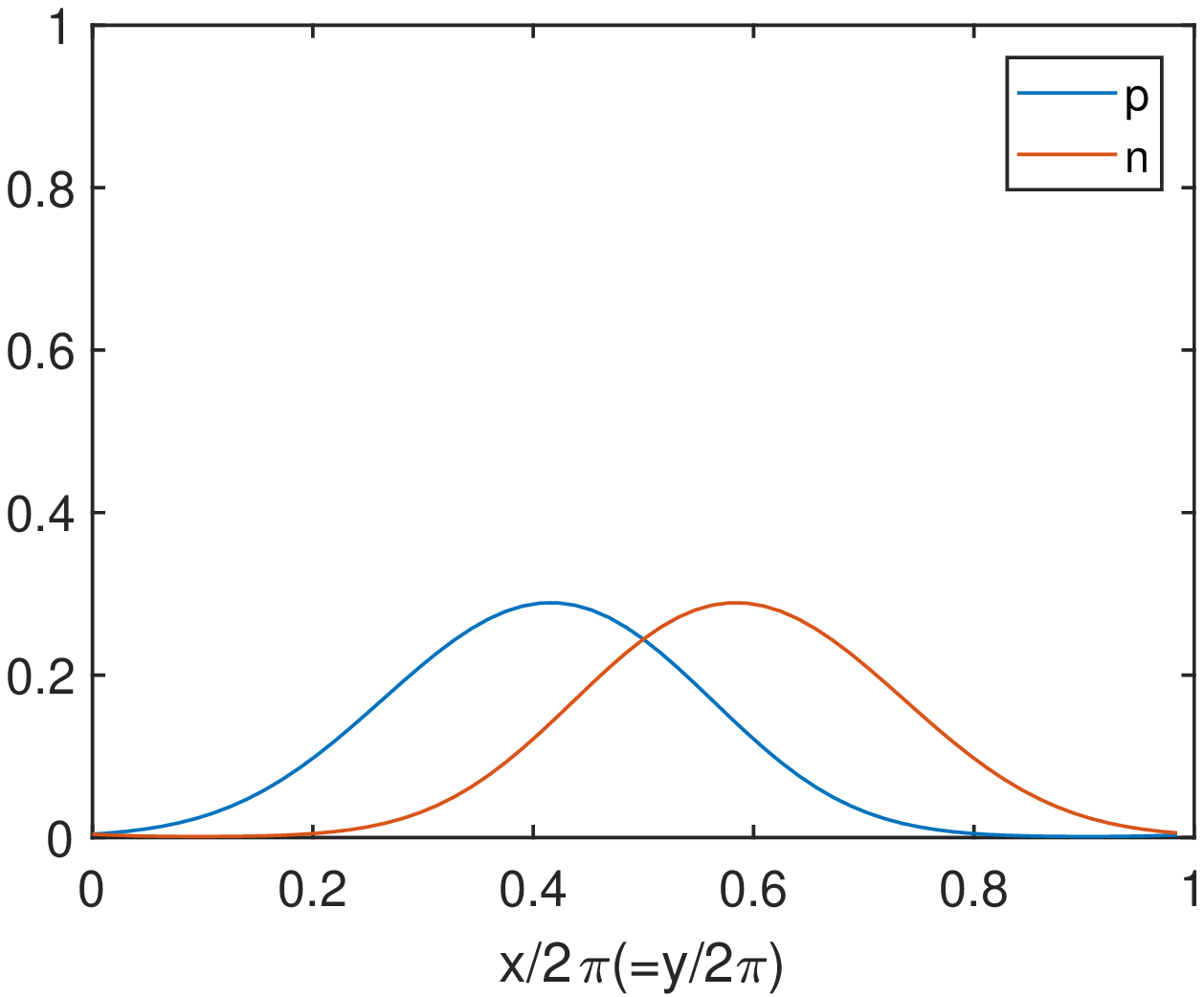}\\
  \includegraphics[width=.45\textwidth,keepaspectratio]{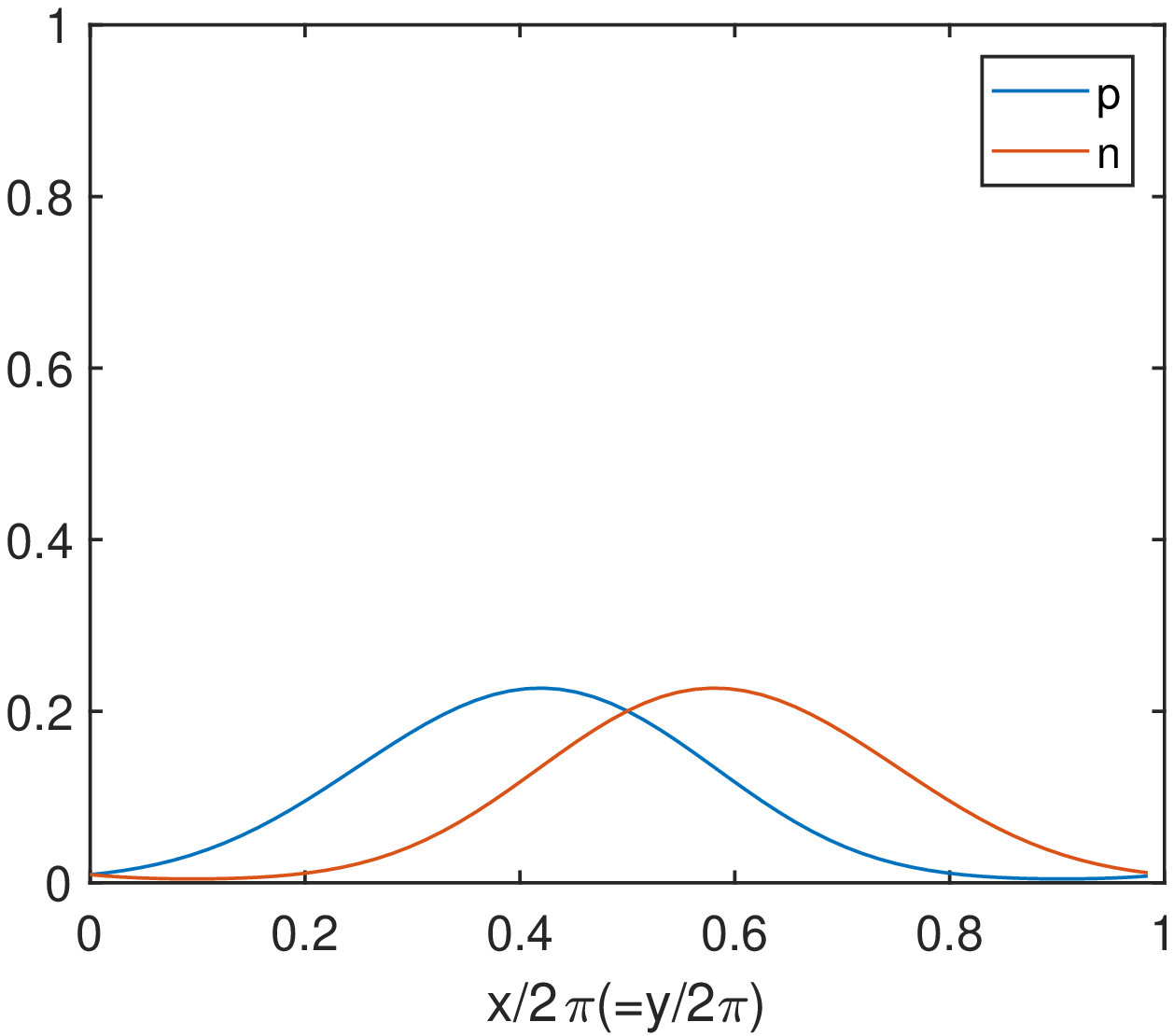}
  \includegraphics[width=.45\textwidth,keepaspectratio]{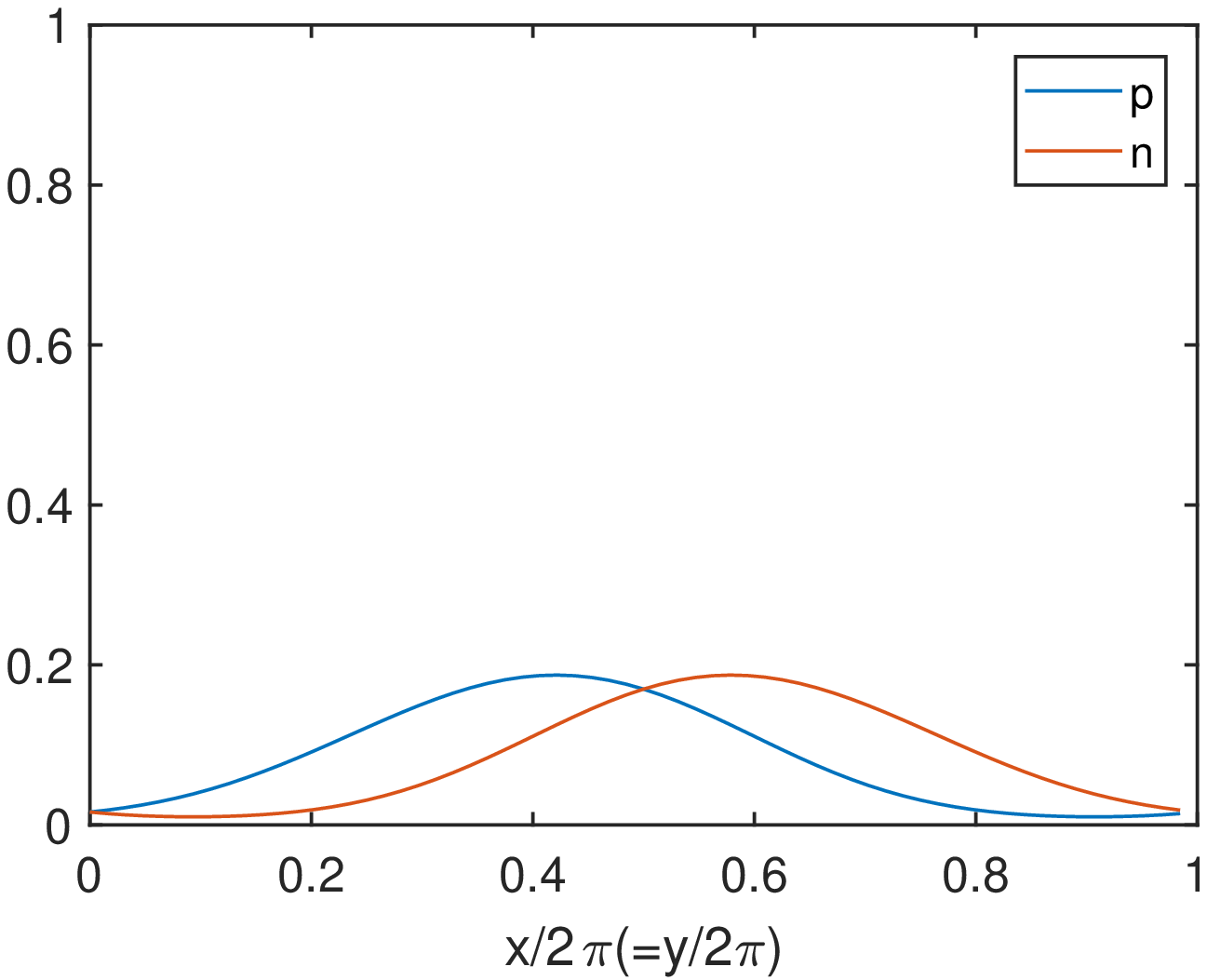}
  \caption{(Example 2) Profiles of $p$ and $n$ on the line $x=y$, at $t=0.2$ (upper-left), $0.6$ (upper-right), $0.8$ (lower-left), $1$ (lower-right). }\label{profile}
\end{figure}

\textbf{Example 2} (Highly disparate initial value). 
The domain, boundary conditions, $\phi_e$, $\epsilon$, and the spatial dicretization are the same as Example 1. 
We choose the initial condition as follows, 
\begin{align*}
  p(x,y,0)&=1+10^{-6}-\tanh\left(2\big((x-0.8\pi)^2+(y-0.8\pi)^2-(0.2\pi)^2\big)\right),\\
  n(x,y,0)&=1+10^{-6}-\tanh\left(2\big((x-1.2\pi)^2+(y-1.2\pi)^2-(0.2\pi)^2\big)\right),
\end{align*}
so that $\min p(x,y,0)=\min n(x,y,0)\approx 10^{-6}$, $\max p(x,y,0)=\max n(x,y,0)\approx 1.65$. 
The initial condition indicates that the positive and negative charged particles accumulates in two regions centered at $(0.8\pi,0.8\pi)$ and $(1.2\pi,1.2\pi)$, respectively. 
By section \ref{twocomponent}, the exact solution satisfies maximum principle and the dissipation of electrostatic potential. 
\begin{figure}
  \centering
 $$ \includegraphics[width=.35\textwidth,keepaspectratio]{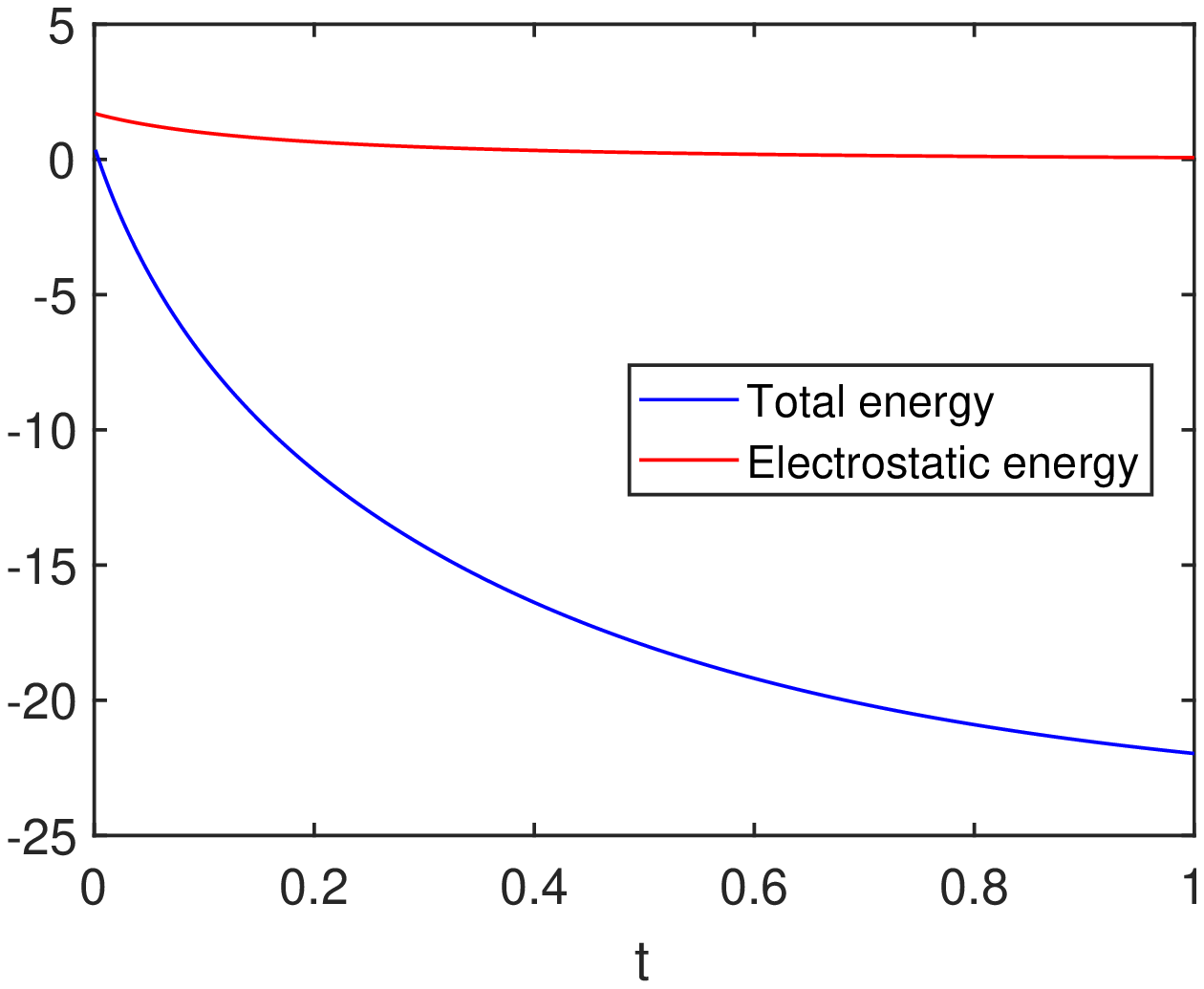}
  \includegraphics[width=.35\textwidth,keepaspectratio]{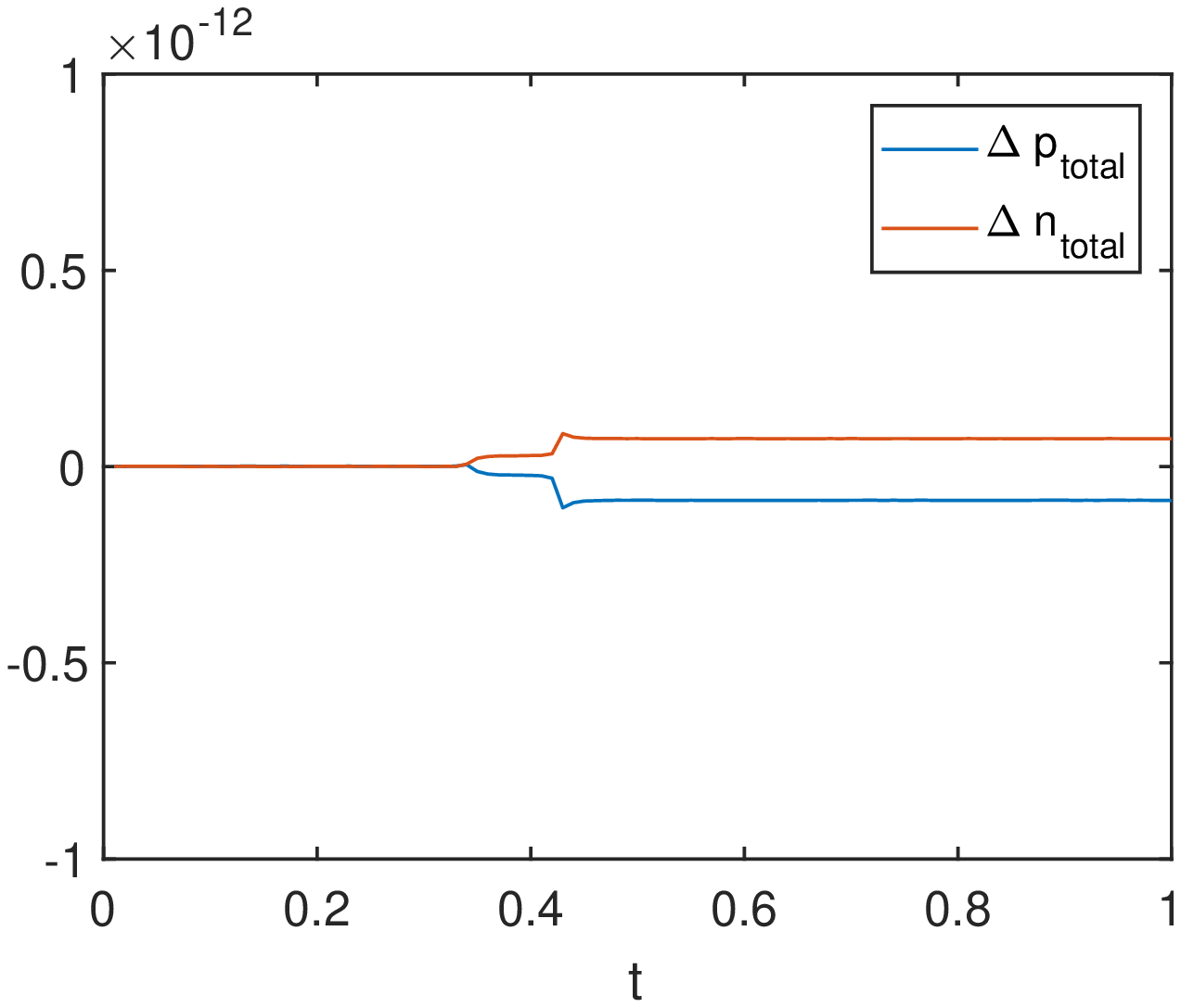}
   \includegraphics[width=.35\textwidth,keepaspectratio]{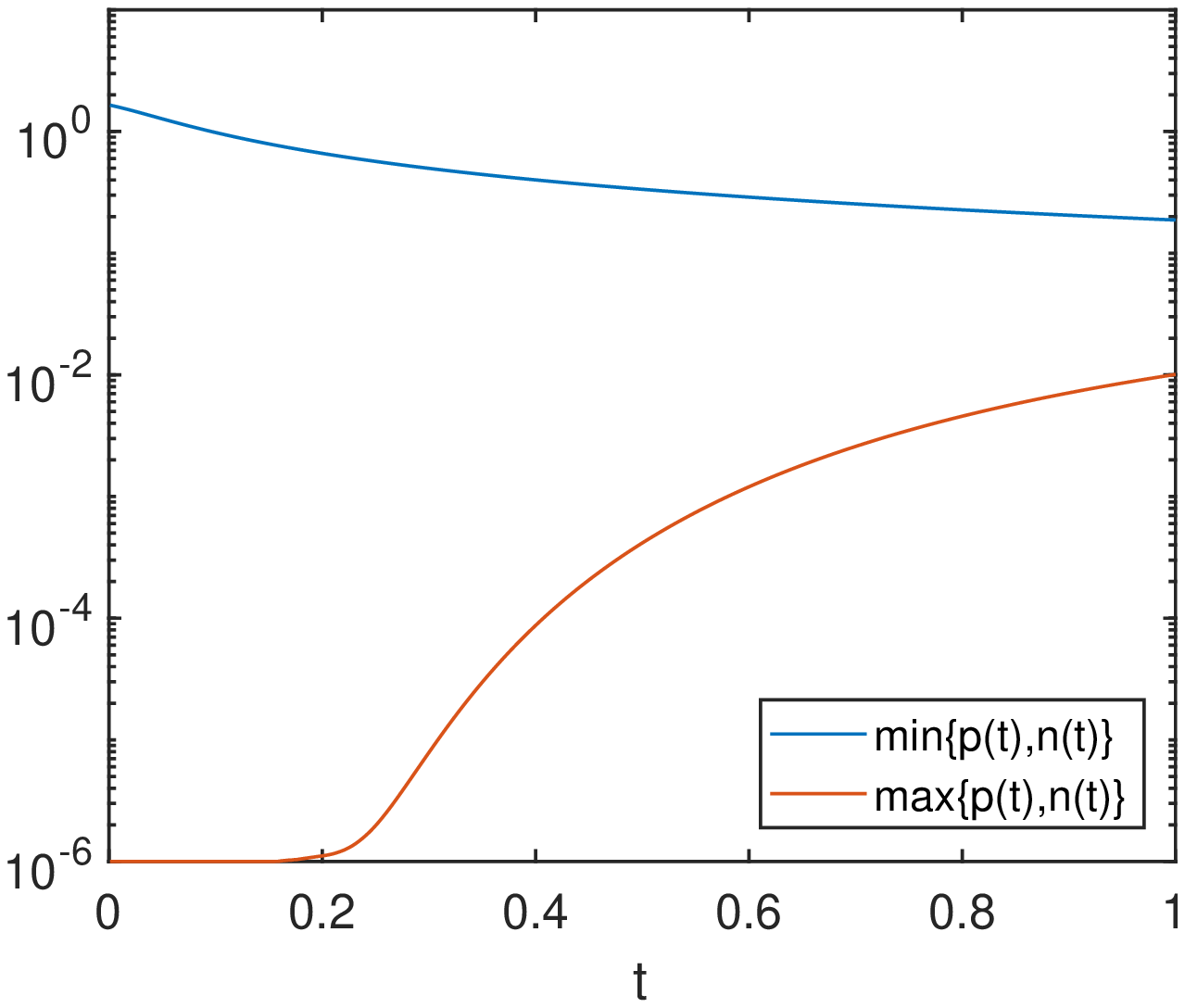}$$
  \caption{(Example 2) Left: total energy density and electrostatic energy density. Middle: deviation of the average concentration to the initial. Right: Lower and upper bound.}\label{eng}
\end{figure}

We use the second-order scheme with the time step $\delta t=10^{-3}$. 
To show the profiles of $p$ and $n$, we plot them on the line $x=y$ at $t=0.1,0.2,0.4,1$ in Fig.~\ref{profile}. 
We also examine the energy dissipation of the total energy and the electrostatic energy in Fig.~\ref{eng} (left), and find they indeed decrease as $t$ grows. 
The change of average concentration is given in Fig.~\ref{eng} (middle), where we find that the error is neglible. 
We also plot the lower and upper bounds of $p$ and $n$ about $t$ in the right of Fig.~\ref{eng}, where we observe that the numerical results keep the maximum principle. 

We also experiment with a larger time step $\delta t=10^{-2}$, where the maximum principle and energy dissipation are still observed. 

\textbf{Efficiency of the scheme. } 
Let us use the Example 2 to examine the efficiency. 
We plot the number of Newton iterations, and the maximum number of the GMRES iteration in each Newton step, for $\delta t=10^{-3}$ and $10^{-2}$. 
The number of the Newton iterations is slightly larger in the first few time steps, and for most time steps we only need 2--4 Newton iterations. 
For the larger time step, one intuitively expects that more Newton iterations are needed, but it turns out that we only need 1--2 more in this example. 
\begin{figure}
  \centering
  \includegraphics[width=.5\textwidth,keepaspectratio]{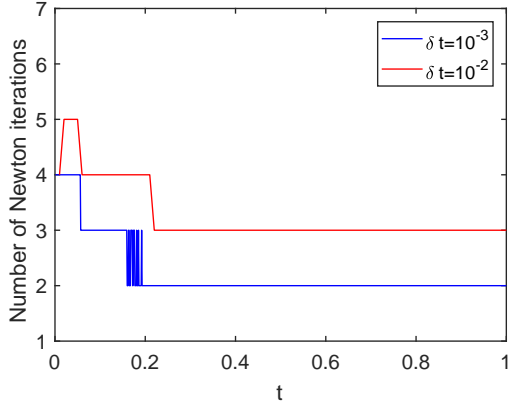}
  \caption{(Example 2) Number of Newton iterations in each time step, for two $\delta t$. }\label{itr}
\end{figure}

\textbf{Effect of boundary values. }
In the following two examples, we solve the PNP equations on $[0,1]^2$. 
The Neumann boundary conditions are imposed on $\mu_i$, while on $\phi$ the Dirichlet boundary conditions are imposed for the four solid line segments, $1/4\le x\le 3/4, y=0,1$ and $1/4\le y\le 3/4, x=0,1$, shown in Fig.~\ref{bnd}. 
For the rest boundary the Neumann boundary conditions are imposed. 
The external potential $\phi_e$ is obtained by solving $-\nabla\cdot(\epsilon\nabla\phi)=0$ with the same types of boundary conditions in Fig.~\ref{bnd}, but nonhomogeneous on the four solid line segments, specified by $\phi_B^L(y),\phi_B^R(y),\phi_B^D(x),\phi_B^U(x)$. 
Recall that for $\phi$ we always assume homogeneous boundary conditions. 
So, it is equivalent to require that the total electric potential $\phi_{total}=\phi+\phi_e$ satisfies 
\begin{align*}
  &-\nabla\cdot(\epsilon\nabla\phi_{total})=\rho_0+\sum_{i=1}^Nz_ic_i,\\
  &\phi_{total}(0,y)=\phi_B^L(y),\ \phi_{total}(1,y)=\phi_B^R(y),\ \frac{1}{4}\le y\le \frac{3}{4}\\
  &\phi_{total}(x,0)=\phi_B^D(x),\ \phi_{total}(x,1)=\phi_B^U(x),\ \frac{1}{4}\le x\le \frac{3}{4}\\
  &\frac{\partial \phi_{total}}{\partial \bm{n}}=0,\quad\text{elsewhere}. 
\end{align*}
\begin{figure}
  \centering
  \includegraphics[width=.48\textwidth,keepaspectratio]{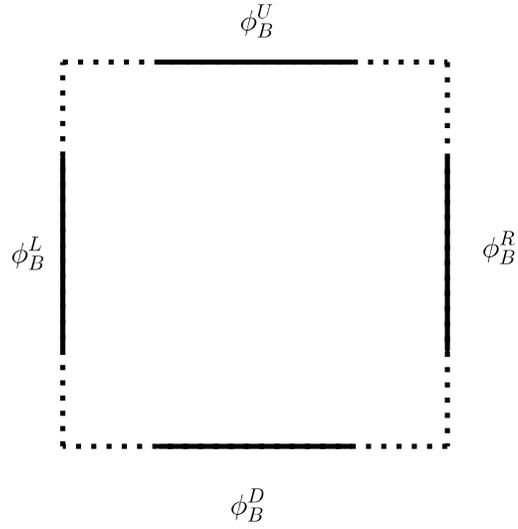}
  \caption{Example 3 \& 4: illustration of boundary conditions on $\phi$. On solid lines Dirichlet boundary conditions are imposed, while on dotted lines Neumann boundary conditions are imposed. }\label{bnd}
\end{figure}

\begin{figure}
  \centering
  \includegraphics[width=.48\textwidth,keepaspectratio]{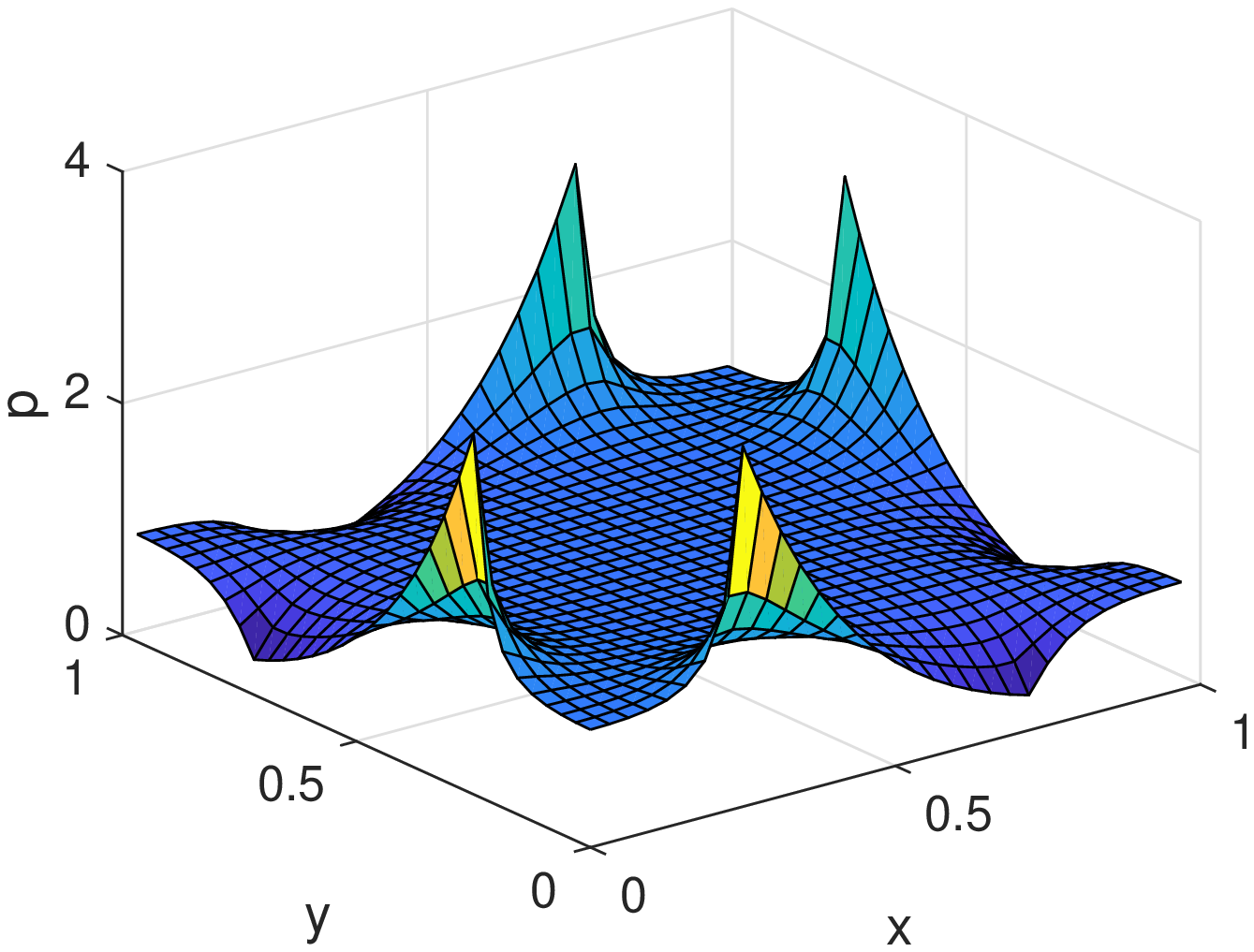}
  \includegraphics[width=.48\textwidth,keepaspectratio]{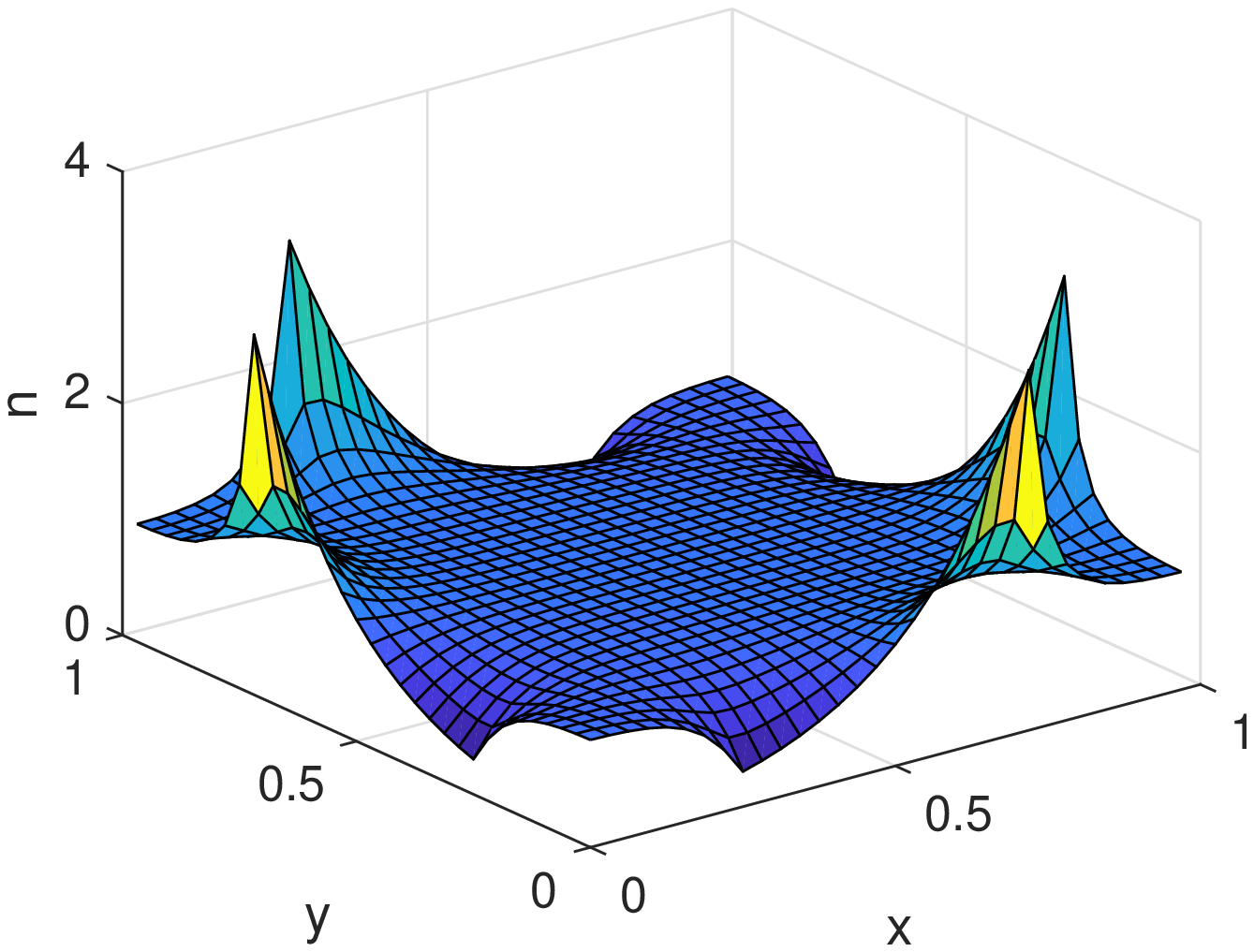}\\
  \includegraphics[width=.48\textwidth,keepaspectratio]{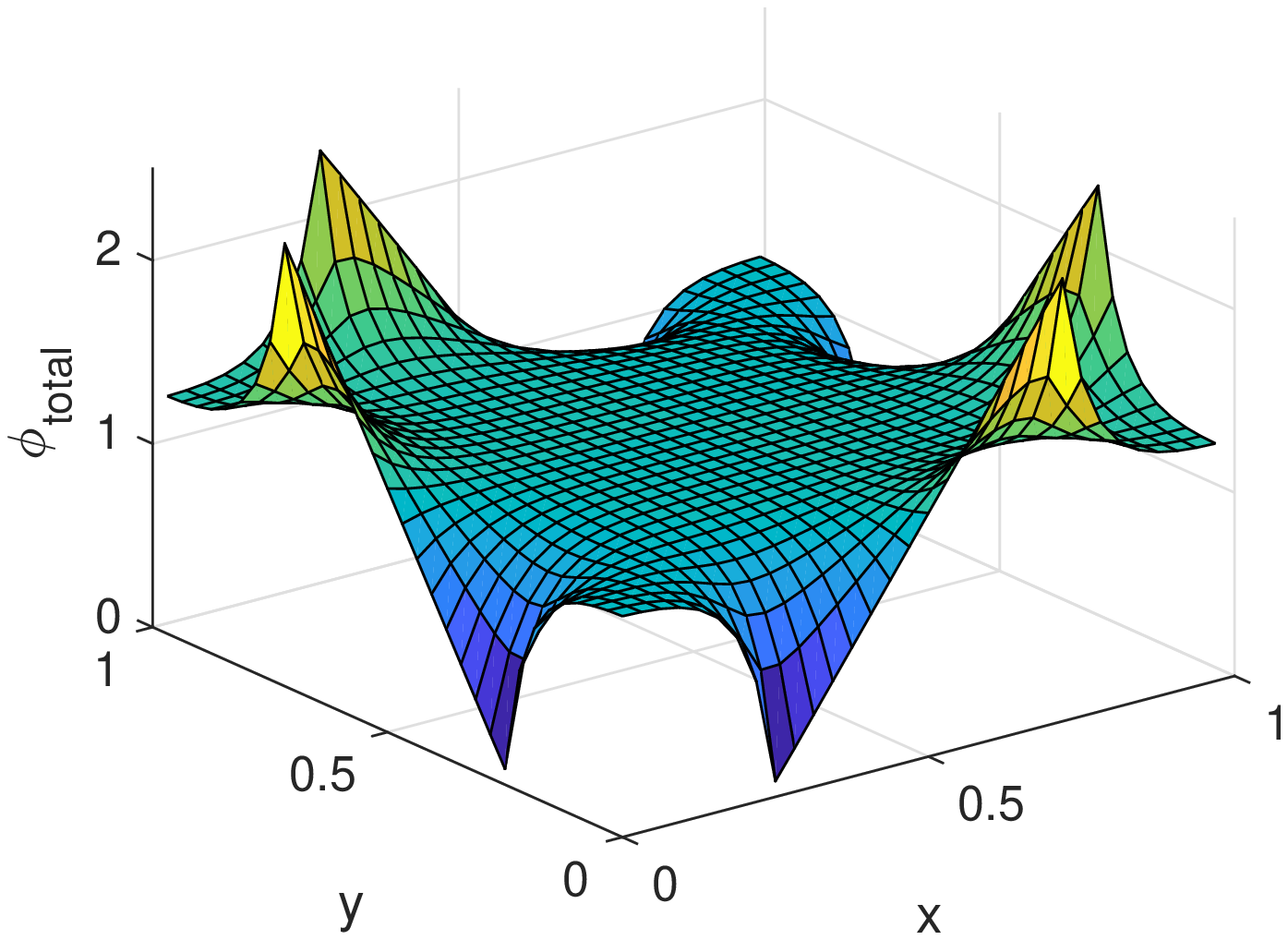}
  \includegraphics[width=.48\textwidth,keepaspectratio]{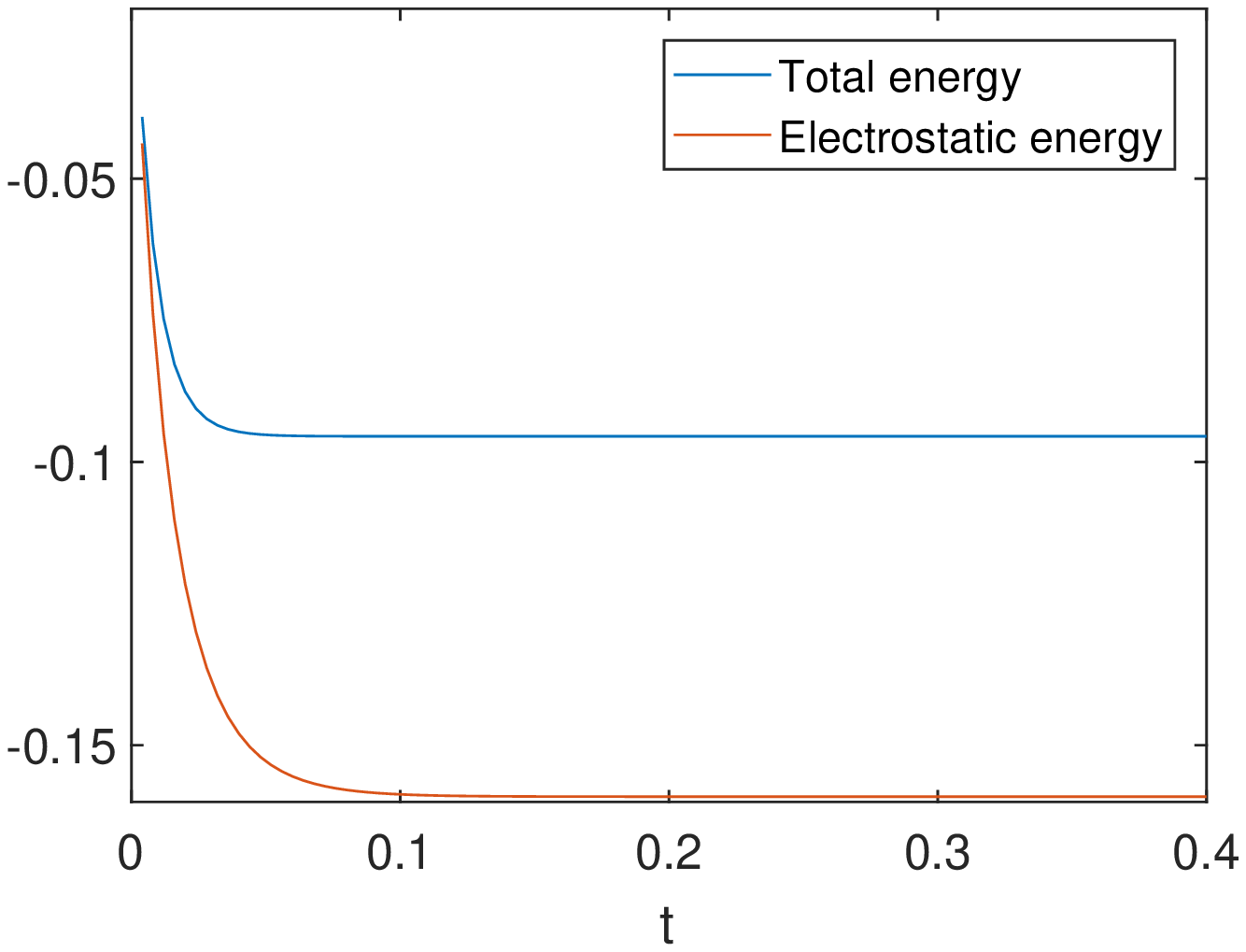}
  \caption{(Example 3) Concentration, eletric potential, and energy for $a=2.5$. }\label{profile_two}
\end{figure}
\textbf{Example 3} (Two-component system with boundary potential). 
We let $z_1=1$, $z_2=-1$, $p=c_1$, $n=c_2$, $D_1=D_2=1$, and $\epsilon=0.01$, $\rho_0=0$. 
The initial value is chosen as $p(x,y,0)=n(x,y,0)=1$. 
The boundary values are specified as follows, 
\begin{align*}
  \phi_B^L(y)=a(y-\frac{1}{4}), \phi_B^R(y)=a(\frac{3}{4}-y),
  \phi_B^D(x)=a(x-\frac{1}{4}), \phi_B^U(y)=a(\frac{3}{4}-x), 
\end{align*}
where $a$ is a parameter to be varied. 
We discretize the space using finite difference method with $32\times 32$ points, and solve the first-order scheme with the time step $\delta t=4\times 10^{-3}$. 
The system reaches steady state after running $100$ steps to $t=0.4$. 

For $a=2.5$, we plot $p$, $n$, $\phi_{total}$ in Fig.~\ref{profile_two}. 
They are mostly flat except near the boundary, with $p$ peaking where $\phi_{total}$ reaches minimum on the boundary, $n$ peaking where $\phi_{total}$ reaches maximum on the boundary. 
Actually, the profile of $n$ is identical to the profile of $p$ rotated by 90 degrees due to the symmetry of the boundary values on $\phi_{total}$. 
The total energy and electrostatic energy are also plotted in Fig.~\ref{profile_two}, where the electrostatic energy here is defined by 
$$
\int \phi_{total}(p-n)\md x\md y. 
$$
Both of them show dissipation, although for the latter it is not proved. 
We also examine how the maximum and minimum concentration evolve with different value of $a$. 
Because of the symmetry, we only plot $p$ in Fig.~\ref{maxmin_boundary}. 
\begin{figure}
  \centering
  \includegraphics[width=.48\textwidth,keepaspectratio]{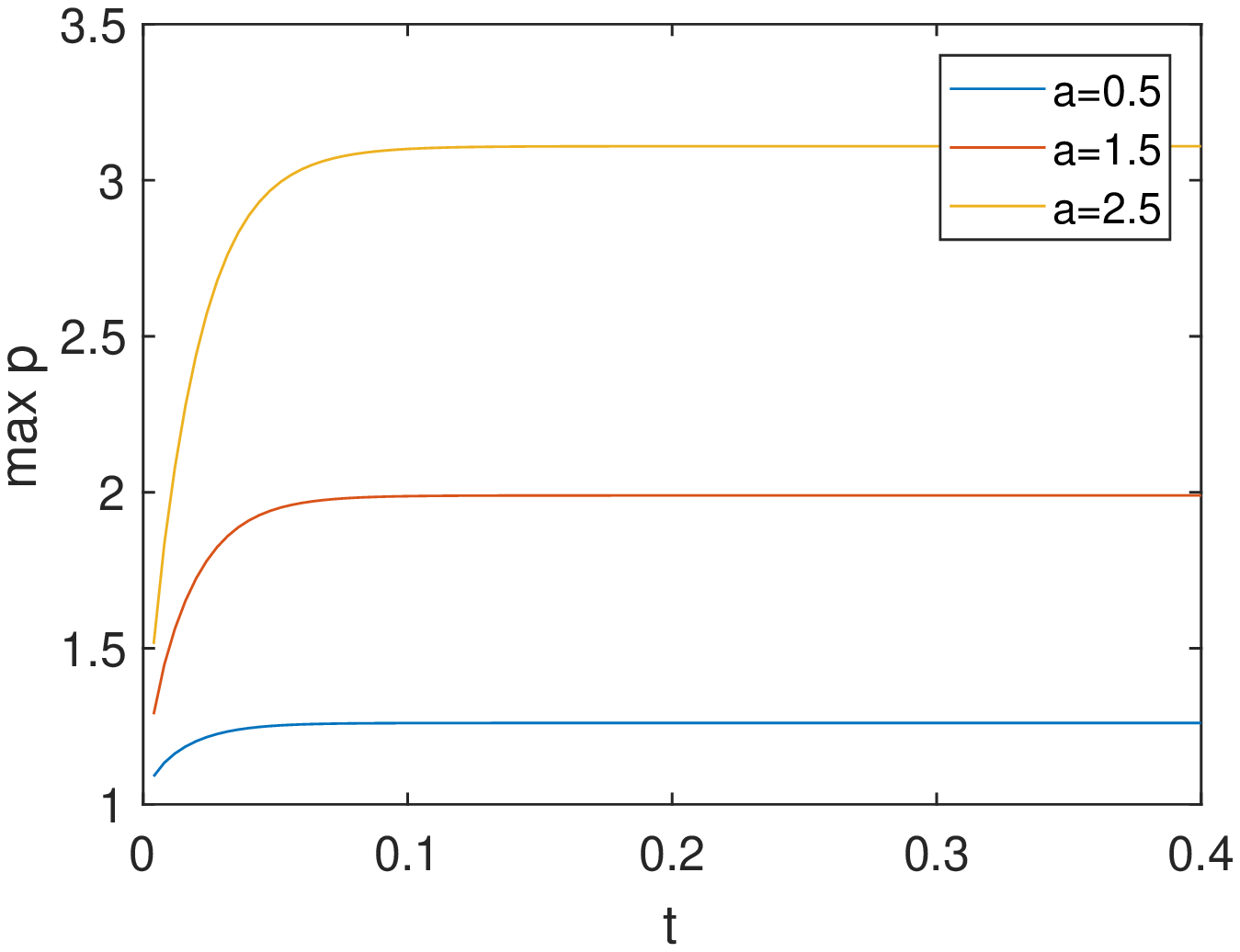}
  \includegraphics[width=.48\textwidth,keepaspectratio]{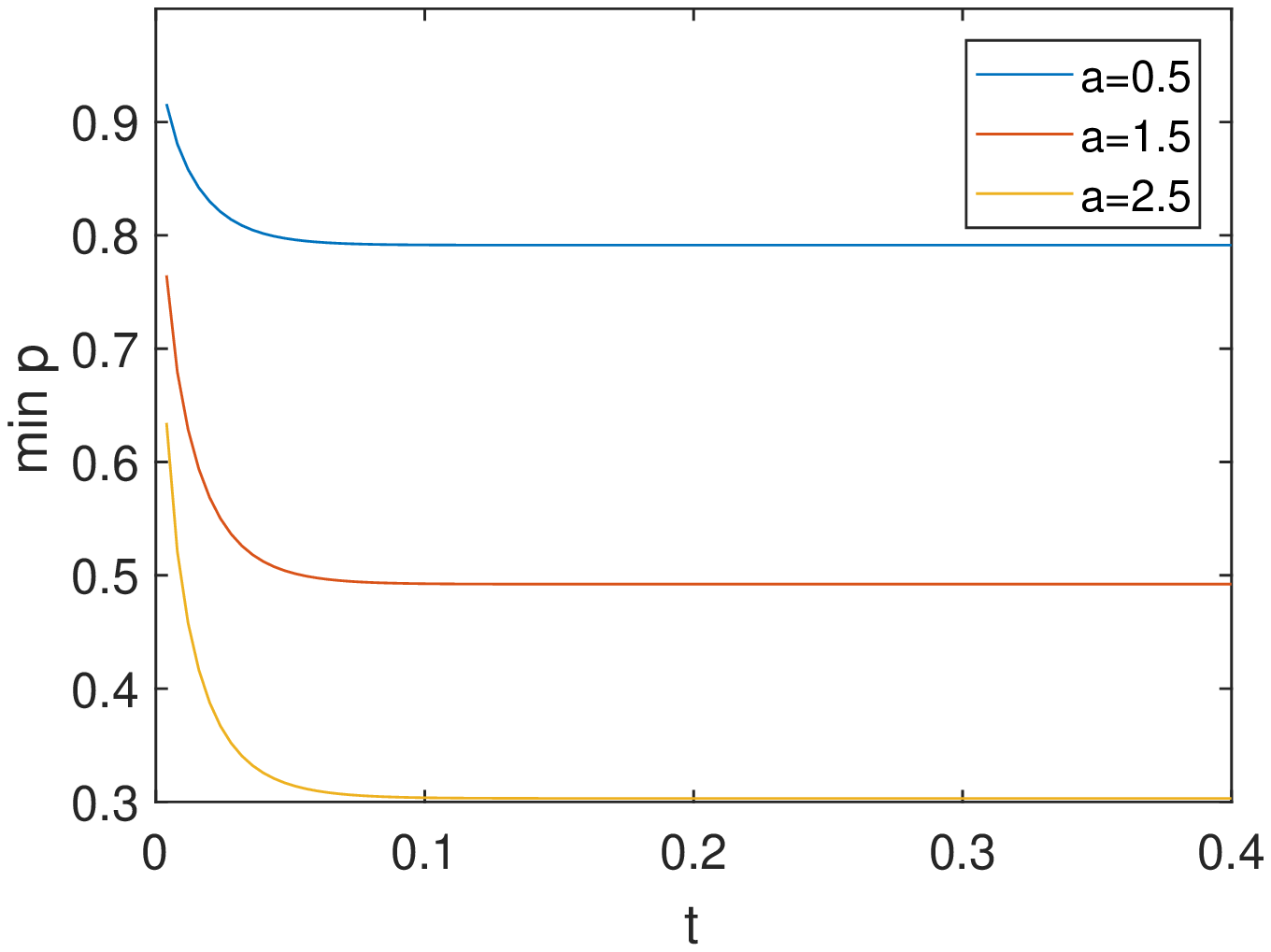}
  \caption{(Example 3) Maximum and minimum concentration for different $a$. }\label{maxmin_boundary}
\end{figure}

\textbf{Example 4.} 
As the last example, we  consider a three-component system.
Choose $z_1=1$, $z_2=-1$, $z_3=2$, and $D_1=D_2=D_3=1$. 
The other settings are identical to Example 3. 
The initial value is chosen as $c_1(x)=c_3(x)=1$ and $c_2(x)=3$ so that the system is electrically neutral. 
The boundary values are chosen as constants on each line segments: 
$$
\phi_B^L=\phi_B^R=-A, \phi_B^D=\phi_B^U=A. 
$$
The spatial and time discretization are also identical to Example 3. 

\begin{figure}
  \centering
  \includegraphics[width=.48\textwidth,keepaspectratio]{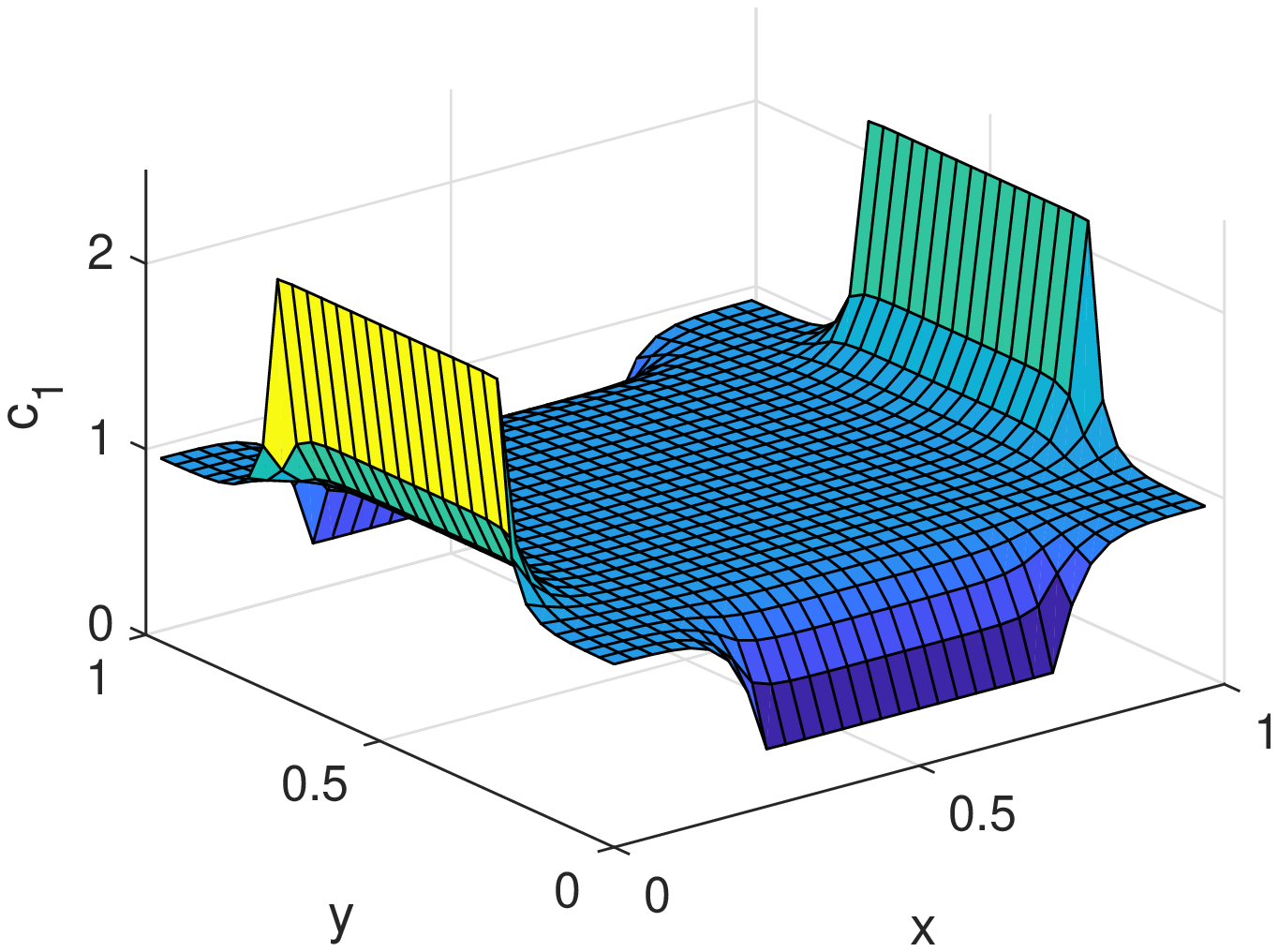}
  \includegraphics[width=.48\textwidth,keepaspectratio]{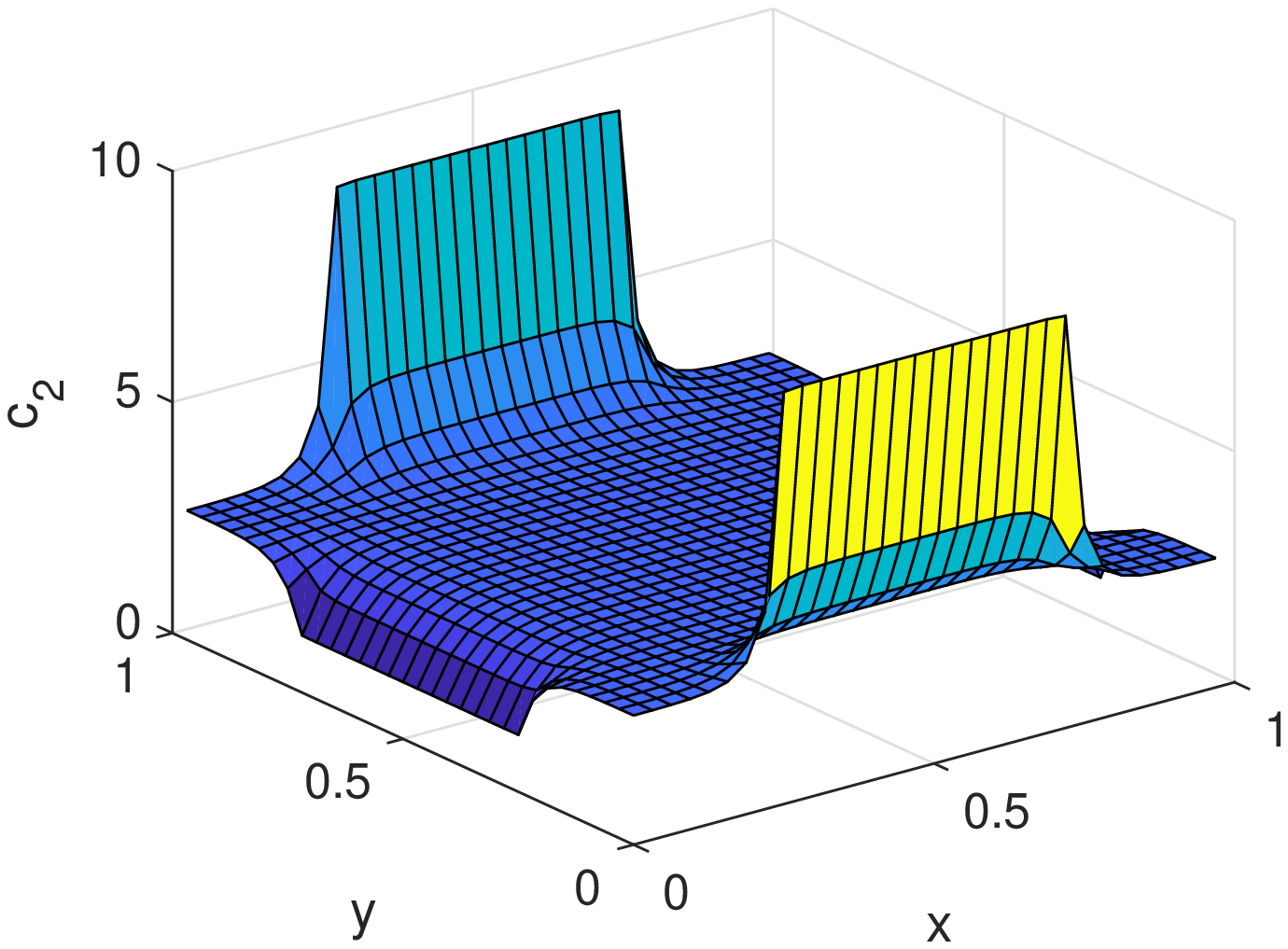}\\
  \includegraphics[width=.48\textwidth,keepaspectratio]{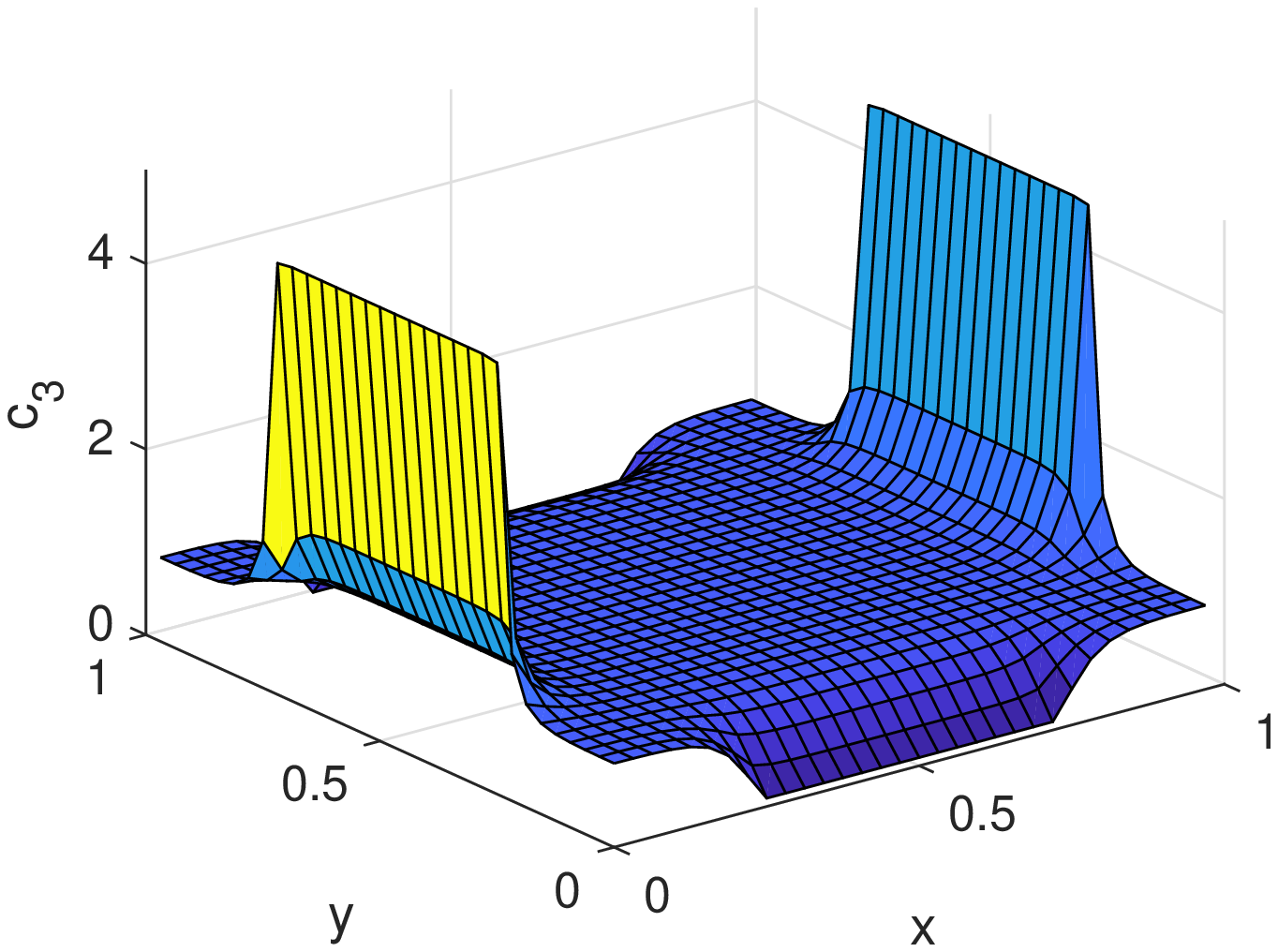}
  \includegraphics[width=.48\textwidth,keepaspectratio]{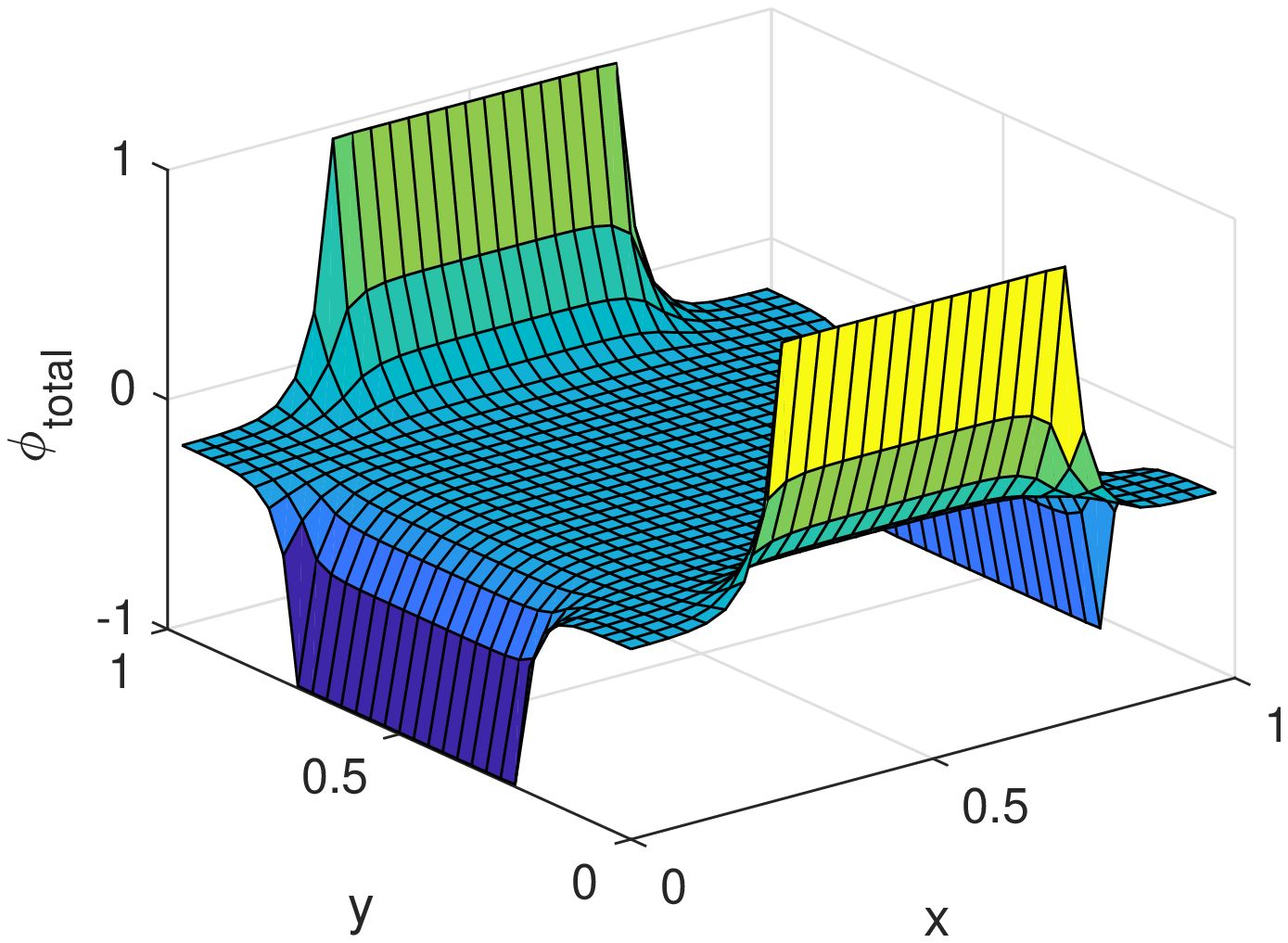}
  \caption{(Example 4) Concentration and eletric potential for $A=1$. }\label{profile_three}
\end{figure}

For $A=1$, the concentration and total electric potential are plotted in Fig.~\ref{profile_three}. 
We also find that they are mostly flat except near the boundary. 
The two types of positive particles accumulate at the left and right boundaries, with $c_3$ larger, while the negative particles accumulate at the other two boundaries. 
We also compare the energy dissipation (Fig.~\ref{energy_three}) and the concentration near the boundaries (Fig.~\ref{boundary_three}). 

\begin{figure}
  \centering
  \includegraphics[width=.48\textwidth,keepaspectratio]{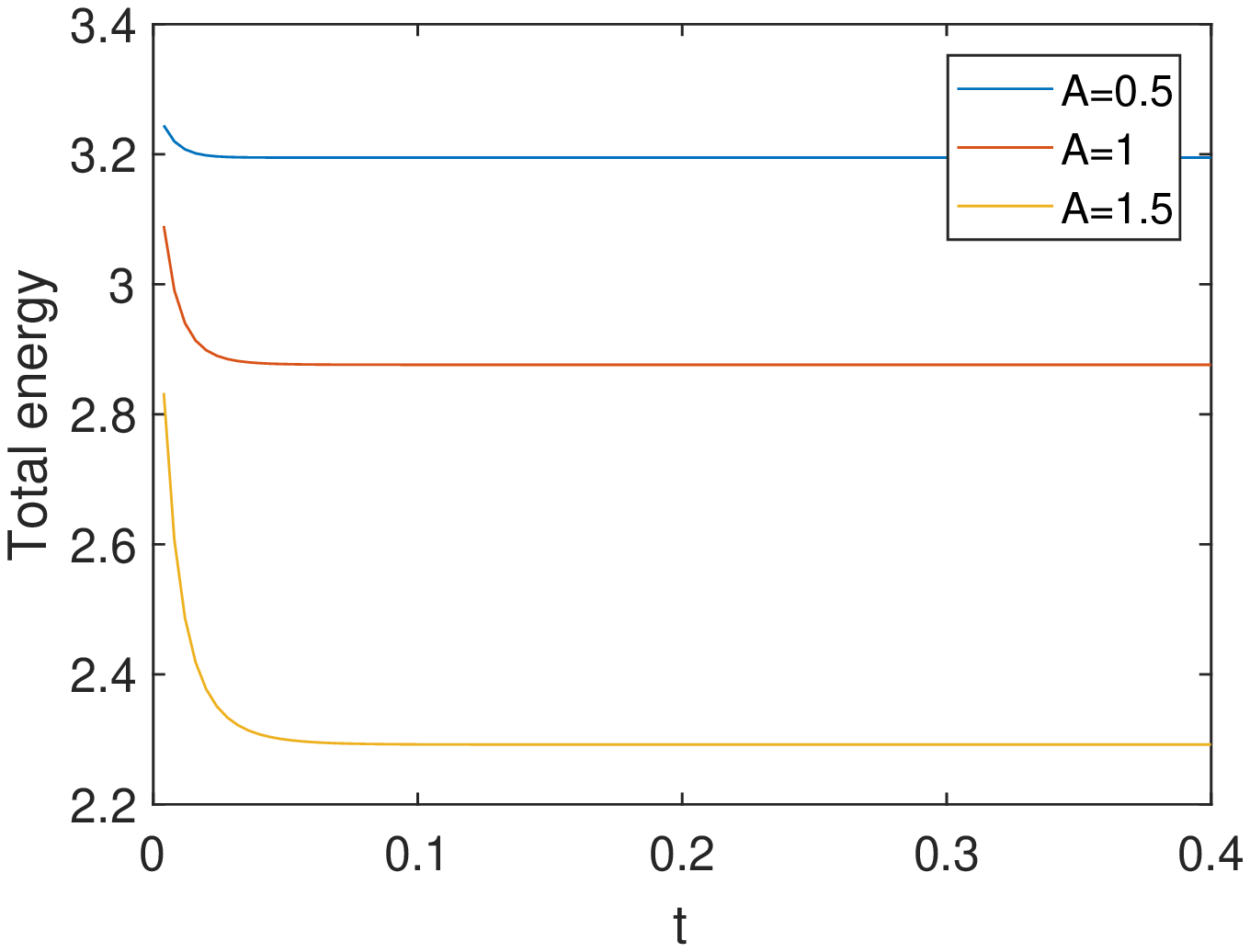}
  \includegraphics[width=.48\textwidth,keepaspectratio]{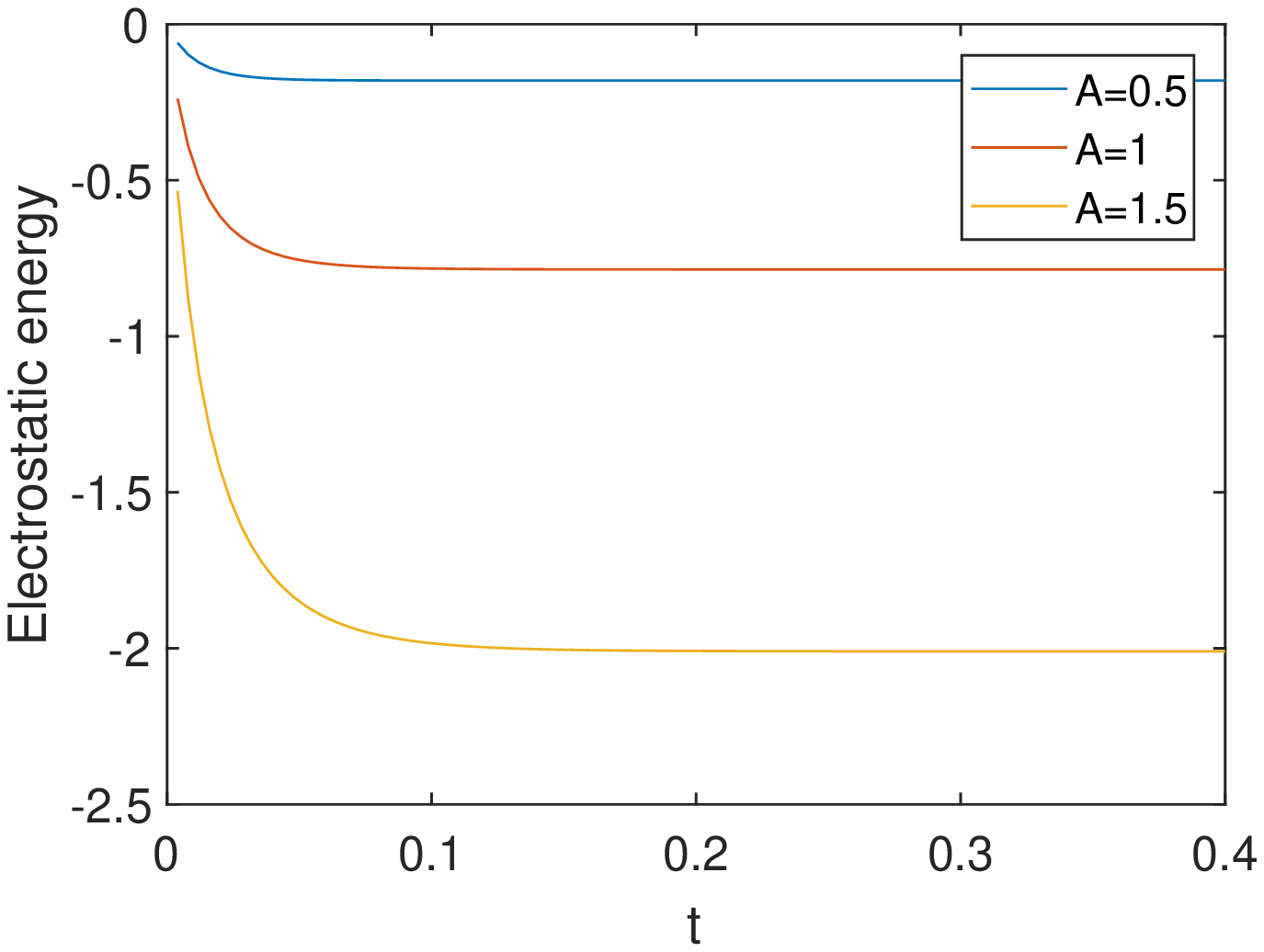}
  \caption{(Example 4) Total energy and electrostatic energy for different boundary values. }\label{energy_three}
\end{figure}

\begin{figure}
  \centering
  \includegraphics[width=.48\textwidth,keepaspectratio]{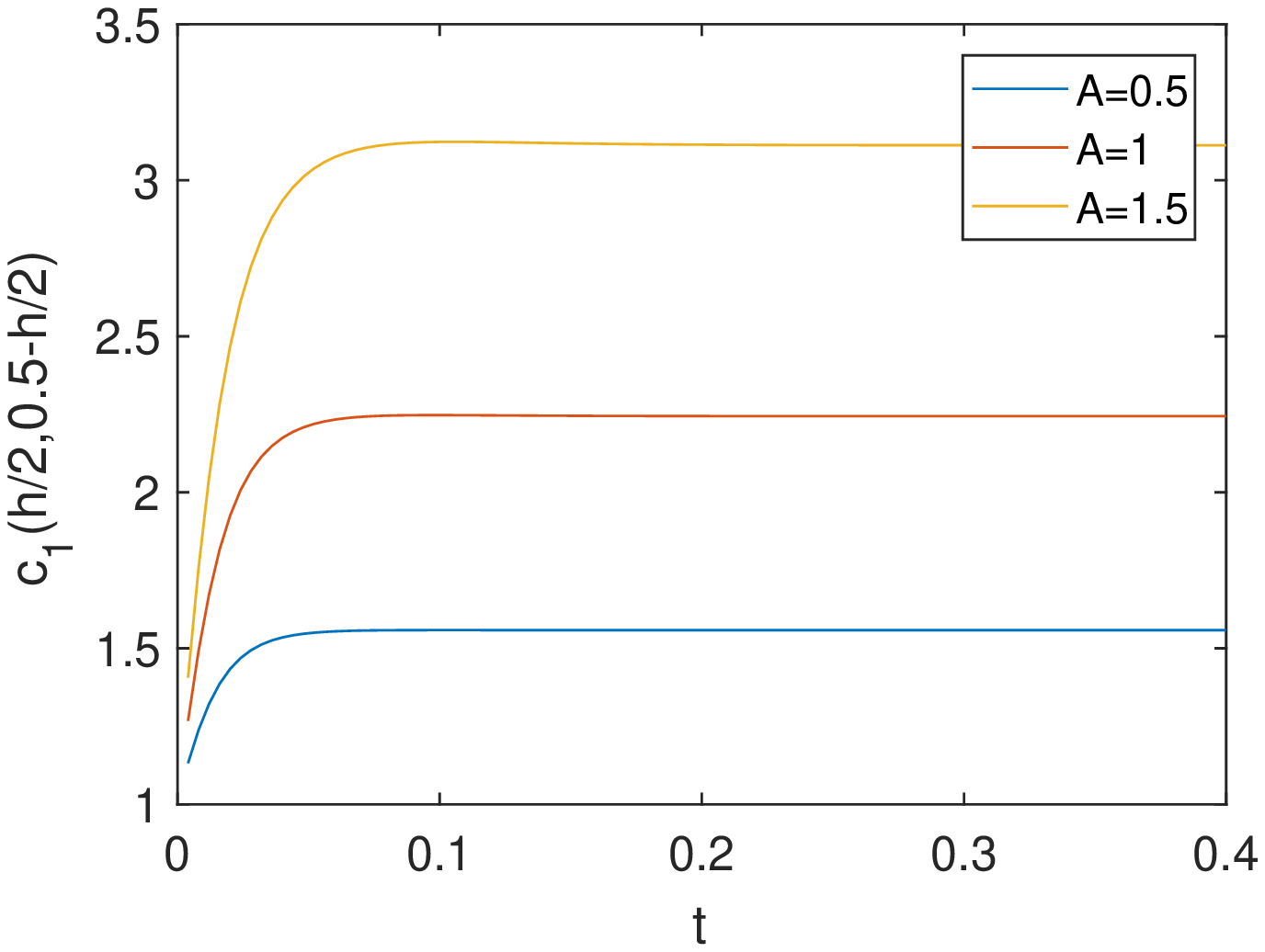}
  \includegraphics[width=.48\textwidth,keepaspectratio]{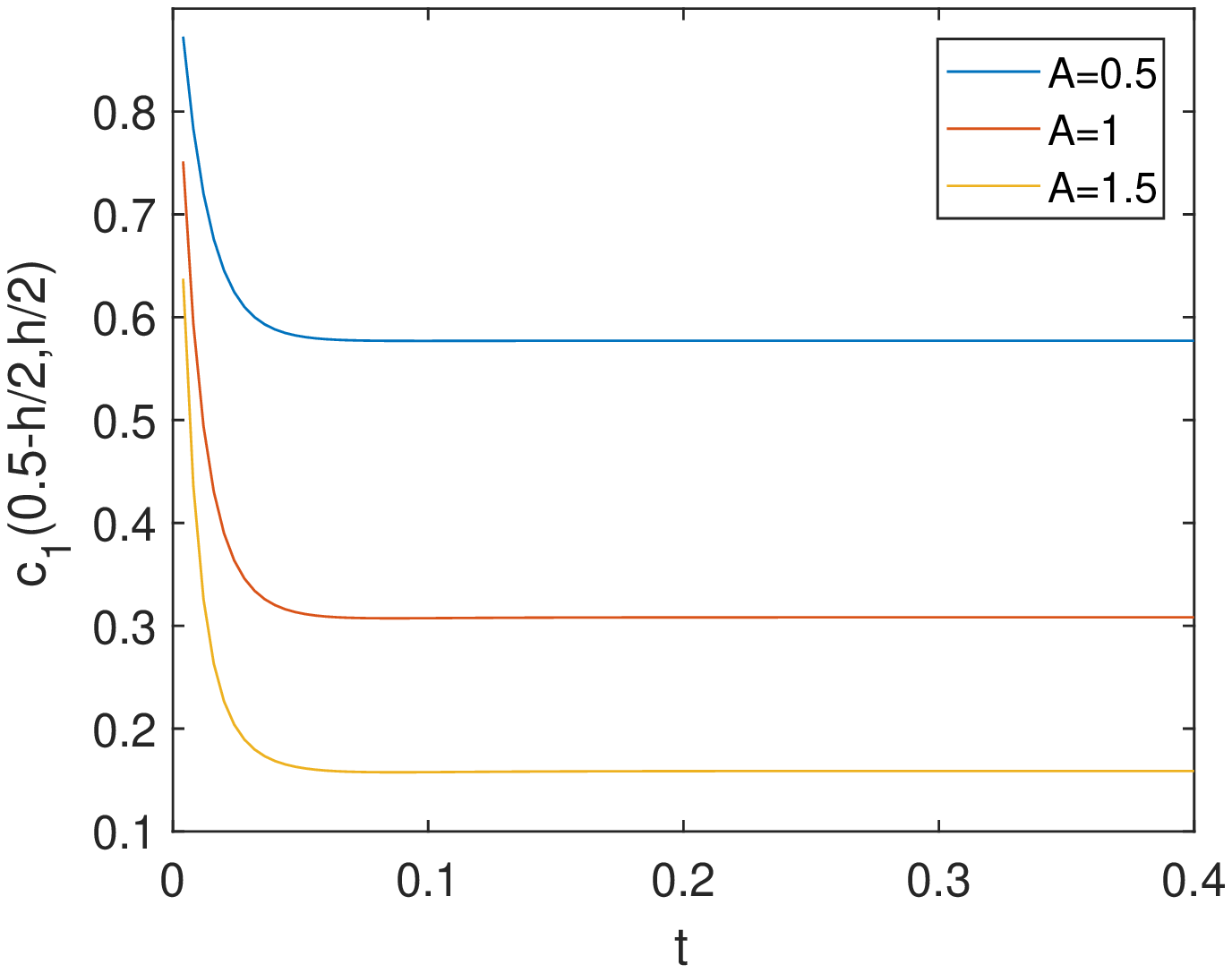}
  \includegraphics[width=.48\textwidth,keepaspectratio]{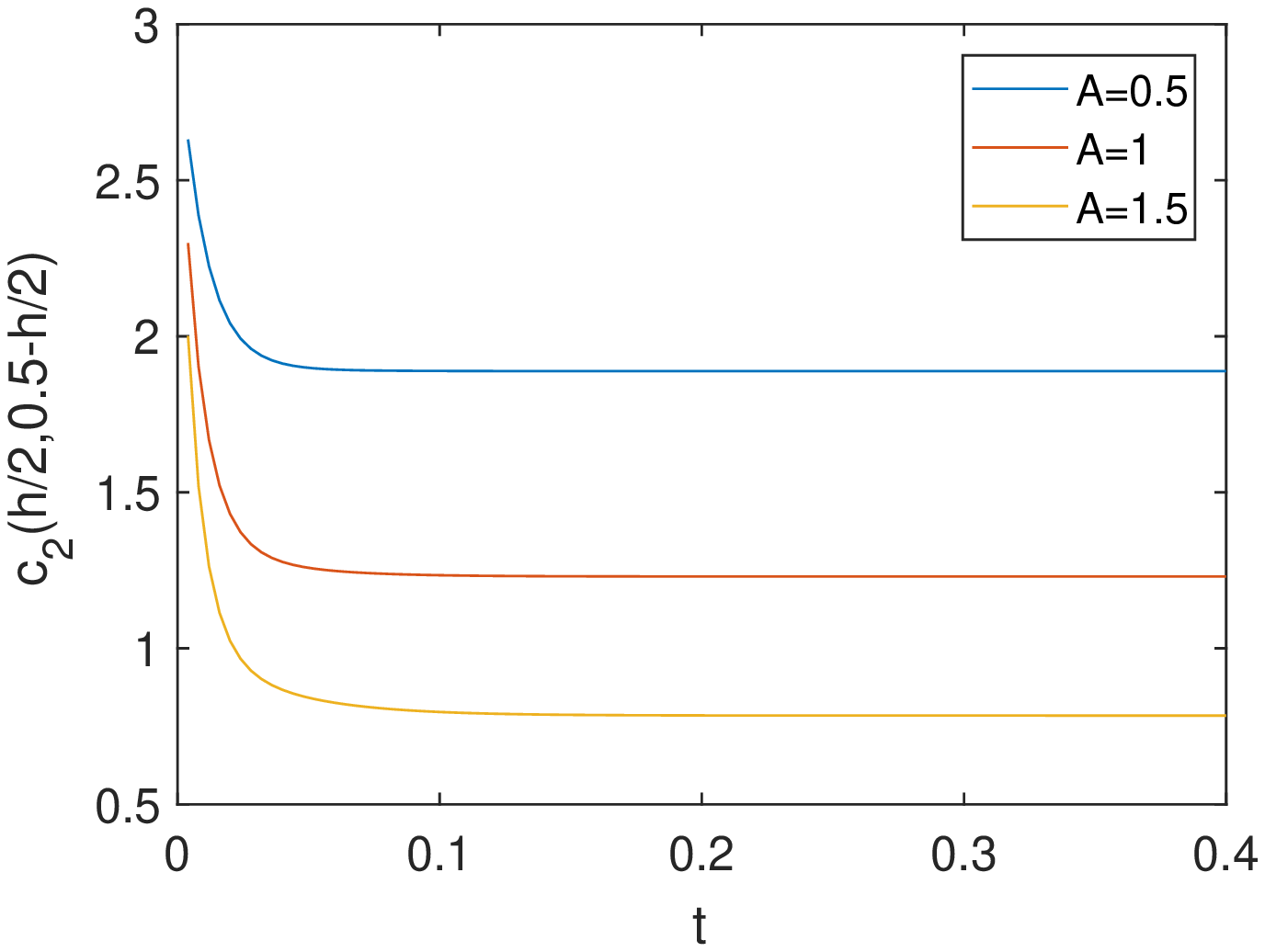}
  \includegraphics[width=.48\textwidth,keepaspectratio]{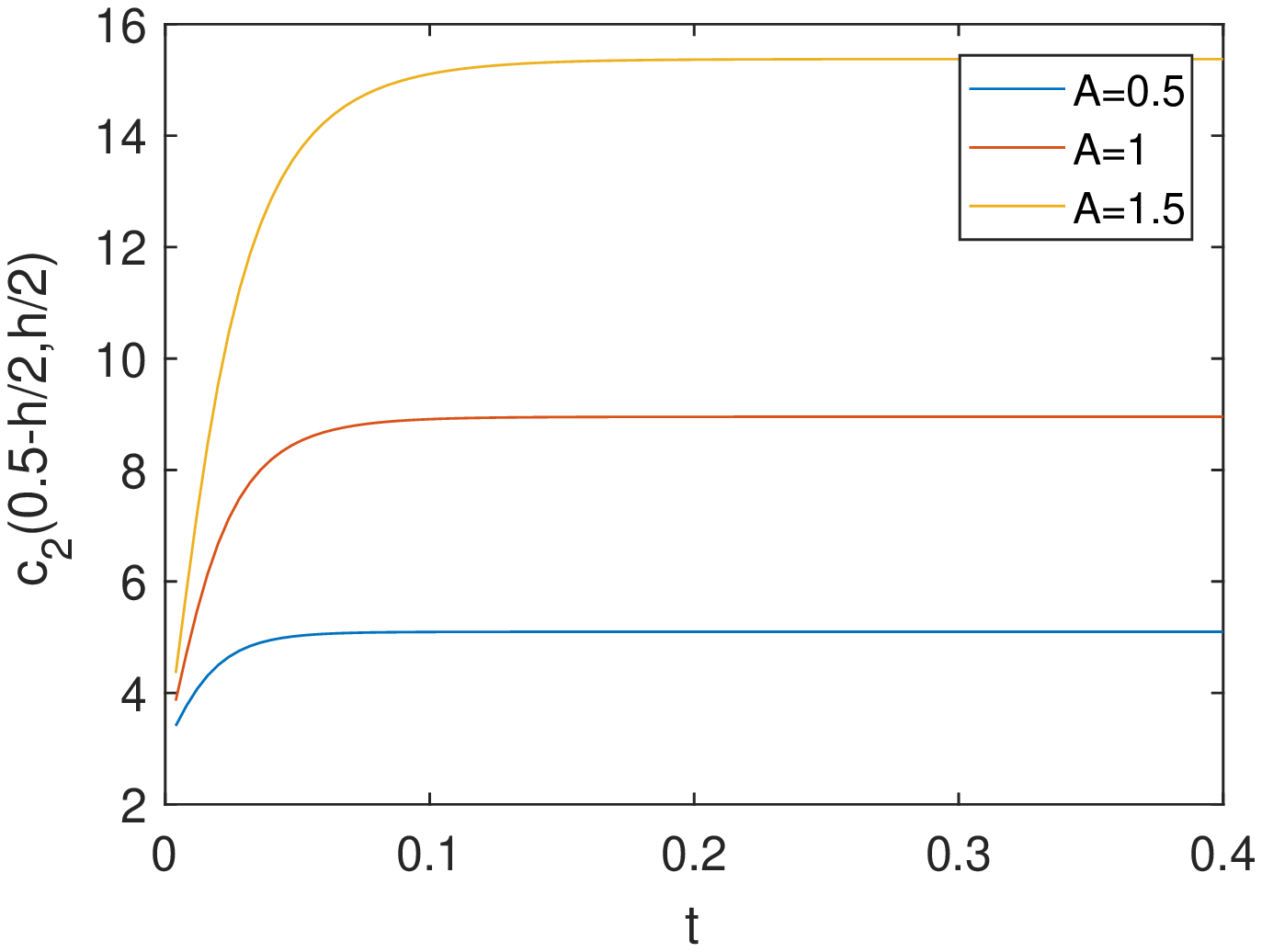}
  \includegraphics[width=.48\textwidth,keepaspectratio]{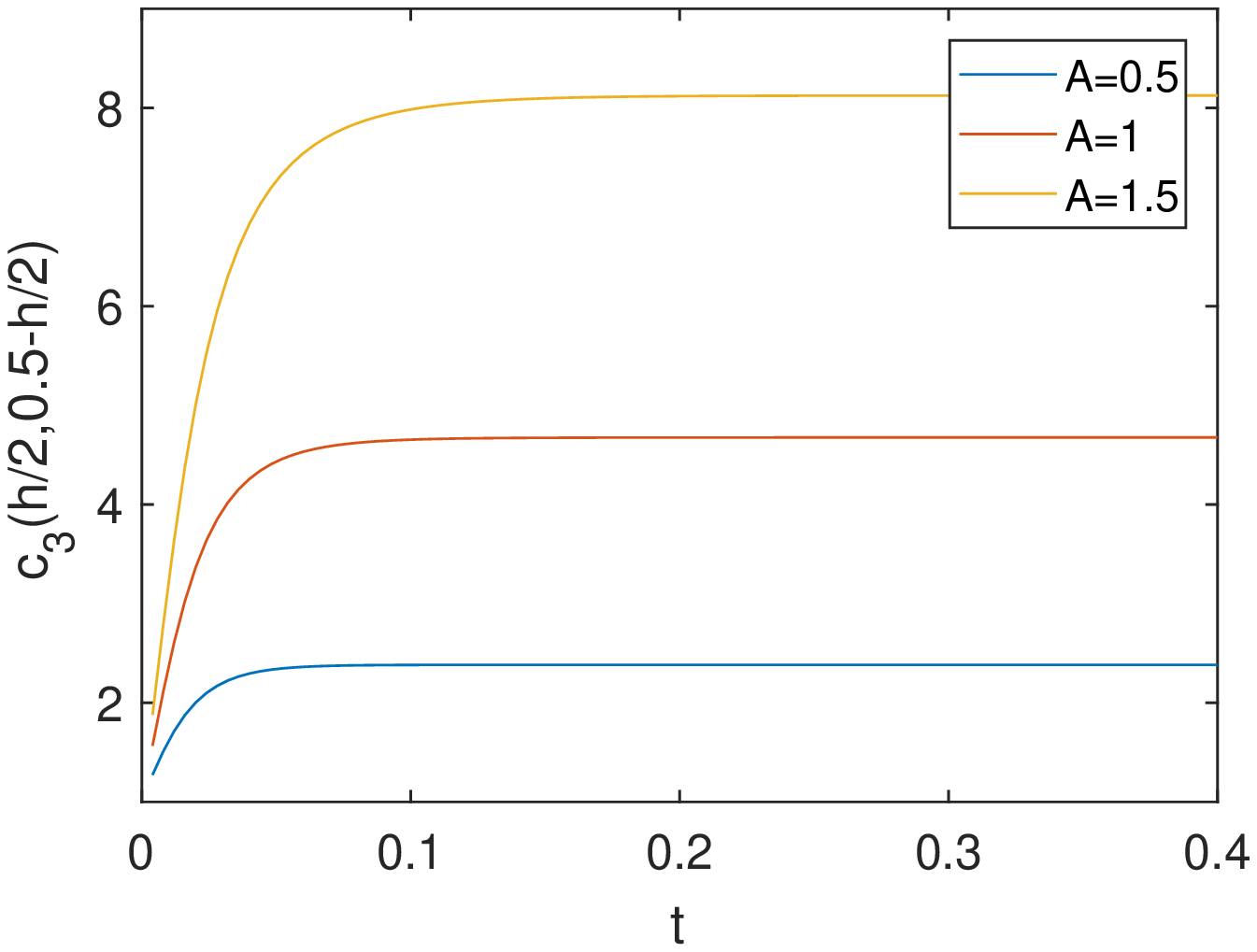}
  \includegraphics[width=.48\textwidth,keepaspectratio]{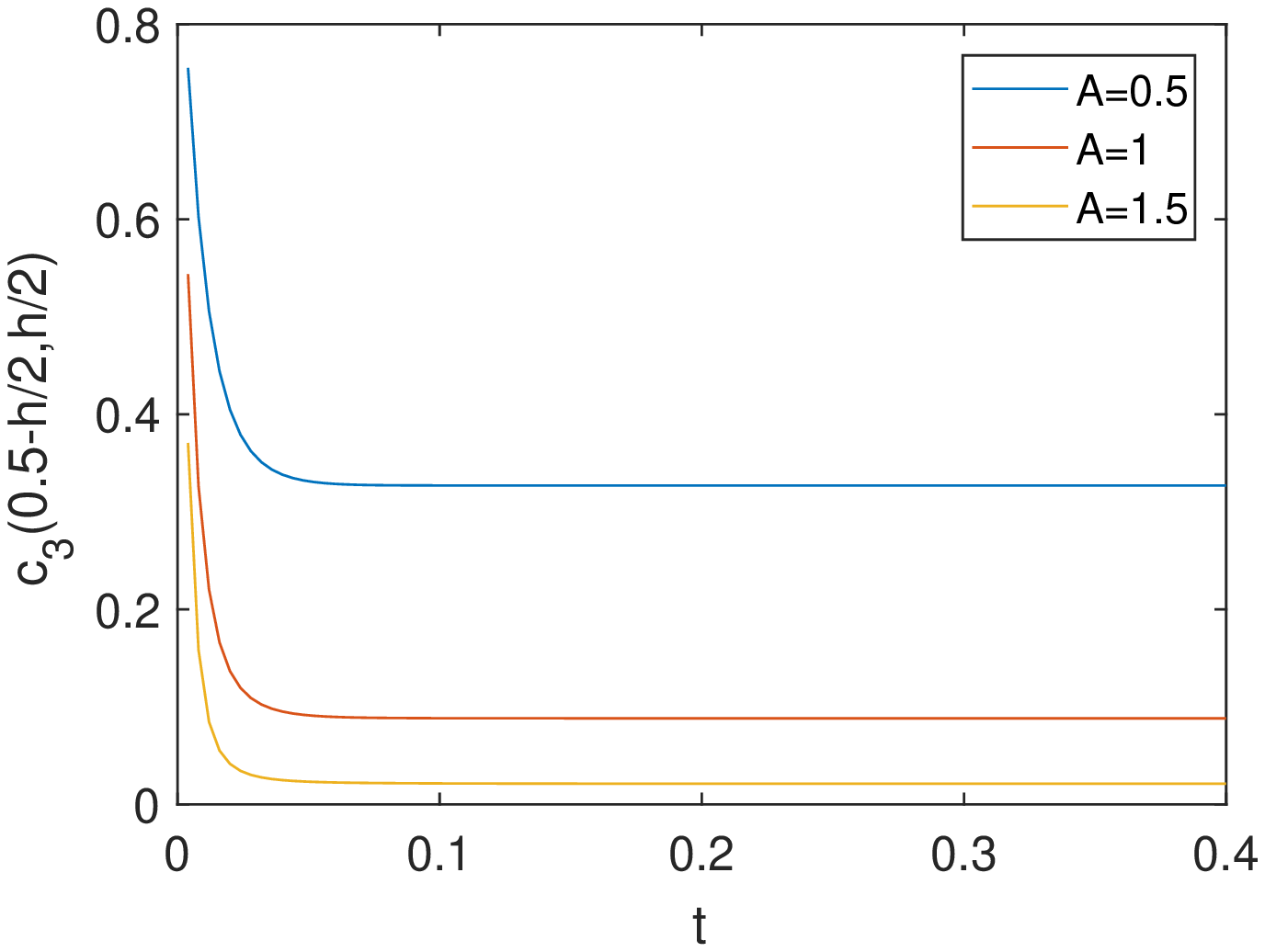}
  \caption{(Example 4) The concentration evolution close to the left (left column) and the upper boundary (right column), for different boundary values. 
    Here we plot the concentration at two grid points $(h/2,0.5-h/2)$ and $(0.5-h/2,h/2)$, near the center of two boundaries, where $h=1/32$. }\label{boundary_three}
\end{figure}

\section{Concluding Remarks}
We proposed in this paper first- and second-order schemes for the PNP equations. 
We proved that both schemes are  unconditionally  mass conservative, uniquely solvable and positivity preserving; and that the first-order scheme is also unconditionally energy dissipative. To the best of our knowledge, our first-order scheme  is  the first such scheme which possesses, unconditionally, all four important properties satisfied by the PNP equations.  While we can not prove the energy dissipation for the second-order scheme, our numerical result indicates that it is energy dissipative as well. 

The schemes lead to nonlinear system at each time step but it possesses a unique solution which is the minimizer of a strictly convex functional. Hence, its solution can be efficiently obtained by using a Newton's iteration method.
We presented ample numerical tests to verify the claimed  properties for both first- and second-order schemes. 
In addition, in special cases where the PNP equation possesses maximum principle and electrostatic energy dissipation, our numerical results show that the schemes also satisfies them.

\bibliographystyle{plain}
\bibliography{bib_gflow}

\end{document}